\documentclass[10pt,reqno,a4paper]{article} 

\usepackage{anysize}
\marginsize{3.0cm}{3.0cm}{2.2cm}{2.2cm}

\usepackage[english]{babel}
\usepackage[centertags]{amsmath}
\usepackage{amsfonts,amssymb,amsthm}
\usepackage[svgnames]{xcolor}
\usepackage{comment}
\usepackage{hyperref} 
\usepackage{mathrsfs}
\usepackage[sans]{dsfont}
\usepackage{accents}

\let\OLDthebibliography\thebibliography
\renewcommand\thebibliography[1]{
  \OLDthebibliography{#1}
  \setlength{\parskip}{0pt}
  \setlength{\itemsep}{0pt plus 0.3ex} }

\numberwithin{equation}{section}

\theoremstyle{plain}
\newtheorem{theorem}{Theorem}[section]
\newtheorem{proposition}[theorem]{Proposition}
\newtheorem{lemma}[theorem]{Lemma}

\theoremstyle{definition}
\newtheorem{remarkbase}[theorem]{Remark}
\newenvironment{remark}{\pushQED{\qed} \remarkbase}{\popQED\endremarkbase}

\renewcommand{\Im}{\mathrm{Im}\,}
\renewcommand{\Re}{\mathrm{Re}\,}

\newcommand{\N}{{\mathbb N}}
\newcommand{\R}{{\mathbb R}}

\newcommand{\Z}{\mathbb Z}

\newcommand{\T}{{\mathbb T}}
\renewcommand{\S}{{\mathbb S}}

\newcommand{\mA}{\mathcal{A}}
\newcommand{\mB}{\mathcal{B}}

\newcommand{\mD}{\mathcal{D}}
\newcommand{\mE}{\mathcal{E}}
\newcommand{\mF}{\mathcal{F}}

\newcommand{\mH}{\mathcal{H}}
\newcommand{\mI}{\mathcal{I}}
\newcommand{\mJ}{\mathcal{J}}
\newcommand{\mK}{\mathcal{K}}

\newcommand{\mM}{\mathcal{M}}

\newcommand{\mR}{\mathcal{R}}
\newcommand{\mS}{\mathcal{S}}
\newcommand{\mT}{\mathcal{T}}
\newcommand{\mU}{\mathcal{U}}
\newcommand{\mV}{\mathcal{V}}

\newcommand{\mY}{\mathcal{Y}}

\renewcommand{\a}{\alpha}
\renewcommand{\b}{\beta}
\newcommand{\g}{\gamma}

\newcommand{\e}{\varepsilon}
\newcommand{\ph}{\varphi}
\newcommand{\lm}{\lambda}
\newcommand{\Lm}{\Lambda}

\newcommand{\Om}{\Omega}
\newcommand{\om}{\omega}

\newcommand{\s}{\sigma}

\renewcommand{\th}{\vartheta}

\newcommand{\la}{\langle}
\newcommand{\ra}{\rangle}
\newcommand{\pa}{\partial}
\renewcommand{\div}{\mathrm{div}\,}
\newcommand{\grad}{\nabla}

\newcommand{\pois}{\mM}

\title{Bifurcation from multiple eigenvalues of rotating traveling waves 
on a capillary liquid drop}
\author{\normalsize{Pietro Baldi, Domenico Angelo La Manna, Giuseppe La Scala}}
\date{} 

\begin{document}

\maketitle

\noindent
\textbf{Abstract.}  
We consider the free boundary problem for a liquid drop of nearly spherical shape with capillarity, 
and we study the existence of nontrivial (i.e., non spherical) rotating traveling profiles 
bifurcating from the spherical shape, where the bifurcation parameter is the angular velocity. 
We prove that every eigenvalue of the linearized problem is a bifurcation point, 
extending the known result for simple eigenvalues 
to the general case of eigenvalues of any multiplicity. 
We also obtain a lower bound on the number of bifurcating solutions. 

The proof is based on the Hamiltonian structure of the problem 
and on the variational argument of constrained critical points  
for traveling waves of Craig and Nicholls (2000, SIAM J.\ Math.\ Anal.\ 32, 323-359),
adapted to the nearly spherical geometry; 
in particular, the role of the action functional is played here by the angular momentum with respect to 
the rotation axis.  

Moreover, the bifurcation equation presents a 2-dimensional degeneration, 
related to some symmetries of the physical problem.  
This additional difficulty is overcome 
thanks to a crucial transversality property,  
obtained by using the Hamiltonian structure  
and the prime integrals corresponding to those symmetries by Noether theorem,
which are the fluid mass 
and the component along the rotation axis of the velocity of the fluid barycenter.


\tableofcontents

\section{Introduction and main results} 

We consider the free boundary problem for 
the motion of a drop of incompressible fluid with capillarity. 
The problem is described by the system   
\begin{align}
\pa_t u + u \cdot \grad u + \grad p 
& = 0 \quad \text{in} \ \Om_t, 
\label{dyn.eq.01} \\
\div u 
& = 0 \quad \text{in} \ \Om_t,
\label{div.eq.01}
\\
p & = \sigma_0 H_{\Om_t} \quad \text{on} \ \pa \Om_t, 
\label{pressure.eq.01}
\\
V_t & = \la u , \nu_{\Om_t} \ra \quad \text{on} \ \pa \Om_t,
\label{kin.eq.01}
\end{align}
where 
$\Om_t \subset \R^3$ is the time-dependent, open, bounded region occupied by the fluid,
$\pa \Om_t$ is its boundary,
$u$ is the fluid velocity vector field, 
$p$ is the pressure of the fluid, 
$\sigma_0$ is the capillarity coefficient, 
$H_{\Om_t}$ is the mean curvature of the boundary $\pa \Om_t$, 
$V_t$ is the normal velocity of the boundary, 
$\nu_{\Om_t}$ is the unit outer normal of the boundary, 
and $\la \cdot , \cdot \ra$ denotes the scalar product of vectors in $\R^3$. 
Equations \eqref{dyn.eq.01}, \eqref{div.eq.01} are 
the Euler equations of incompressible fluid mechanics, 
\eqref{pressure.eq.01} gives the pressure at the boundary in terms of capillarity, 
and \eqref{kin.eq.01} is 
the kinematic condition that the movement of the boundary $\pa \Om_t$ 
in its normal direction is due to the movement of the liquid particles on $\pa \Om_t$. 
The unknowns are the domain $\Om_t$, 
the velocity vector field $u$, 
and the pressure $p$.

We assume that the boundary $\pa \Om_t$ is the graph of a function 
over the unit sphere,  
\begin{equation} \label{def:omega}
\pa \Om_t  =\{ (1+h(t,x)) x : x \in \S^2\},
\quad \ 
\S^2 = \{ x \in \R^3 : |x|=1 \},
\end{equation}
where the elevation function $h$ satisfies $1+h > 0$. 
When $|h|$ is small, we say that $\Om_t$ is a nearly spherical domain. 
We also assume that the motion of the fluid is irrotational, 
so that the velocity vector field $u$ is given by the gradient  
of a scalar function, the velocity potential $\Phi$. 
For $u = \grad \Phi$, 
\eqref{dyn.eq.01} becomes the equation
$\grad(\pa_t \Phi + \frac12 |\grad \Phi|^2 + p) = 0$ in $\Om_t$, that is,
$\pa_t \Phi + \frac12 |\grad \Phi|^2 + p =$ independent of $x$
in $\Om_t$,
\eqref{div.eq.01} becomes $\Delta \Phi = 0$ in $\Om_t$,
and \eqref{kin.eq.01} becomes $V_t = \la \grad \Phi , \nu_{\Om_t} \ra$ on $\pa \Om_t$.

As is proved in \cite{Beyer.Gunther, Shao.initial.notes, B.J.LM}, 
this problem admits a Craig-Sulem formulation, 
which is the equivalent system
\begin{equation} \label{syst.pat.h.pat.psi} 
\pa_t h = X_1(h, \psi), \quad \ 
\pa_t \psi \sim X_2(h,\psi),
\end{equation}
where, in the notations of \cite{B.J.LM},  
\begin{align} 
X_1(h,\psi) 
& := \frac{\sqrt{(1+h)^2 + |\grad_{\S^2} h|^2}}{1 + h} \, G(h)\psi,
\label{def.X.1}
\\
X_2(h, \psi) 
& := \frac12 \Big( G(h)\psi + \frac{\la \grad_{\S^2} \psi , \grad_{\S^2} h \ra}{(1+h) \sqrt{(1+h)^2 
	+ |\grad_{\S^2} h|^2}} \Big)^2
- \frac{|\grad_{\S^2} \psi|^2}{2(1+h)^2} 
- \s_0 H(h),
\label{def.X.2}
\end{align}
and $f \sim g$ means that the difference $f - g$ is independent of $x$.
In system \eqref{syst.pat.h.pat.psi} the unknowns $h, \psi$ 
are real valued functions of the variables $(t,x) \in \R \times \S^2$, 
$\grad_{\S^2}$ denotes the tangential gradient on $\S^2$ 
(defined in \eqref{def.grad.S2.Lap.S2} below), 
$H(h)$ is the mean curvature of the boundary $\pa \Om_t$ in \eqref{def:omega} expressed in terms 
of the elevation function $h$ (see formula \eqref{eq:meancurvature}), 
and $G(h)\psi$ is the Dirichlet-Neumann operator defined 
(omitting to indicate the time-dependence of functions and domains) as
\begin{equation} \label{def:Dirichlet-Neumann}
G(h) \psi (x) = \langle (\nabla \Phi)(\gamma(x)), \nu_{\Omega}(\gamma(x)) \rangle
\end{equation}
at all points $\gamma(x) = (1 + h(x)) x \in \pa \Om$, i.e., for all $x \in \S^2$,
where $\Phi:\Omega \to \R$ is the unique solution in $H^1(\Om)$ of the boundary value problem 
\begin{equation}
    \label{def:Dirichlet-Neumann2}
    \Delta \Phi = 0 \,\, \text{in } \, \Omega \quad \text{and} \quad 
		\Phi(\gamma(x)) = \psi(x)  \,\, \text{for all } \, \g(x) \in \pa \Omega, 
		\ \text{i.e., for all } \, x \in \S^2. 
\end{equation}  
We refer to \cite{B.J.LM} for details about the derivation of system \eqref{syst.pat.h.pat.psi}
and about some properties of the Dirichlet-Neumann operator. 
Note that system \eqref{syst.pat.h.pat.psi} is not a free boundary problem, 
as the unknown functions $h, \psi$ are defined on the fixed surface $\S^2$. 

One has $X_k(h, \psi_1) = X_k(h,\psi)$ for all $\psi_1 \sim \psi$, $k=1,2$,
so that \eqref{syst.pat.h.pat.psi} can also be viewed as a system where $h$ is a function 
and $\psi$ is an equivalence class of functions. 
However, in the present paper we find it more convenient to work with functions, 
instead of with equivalence classes. 
For this reason, we consider the system 
\begin{equation} \label{syst.pat.h.pat.psi.2.sigma.0}
\pa_t h = X_1(h, \psi), \quad \ 
\pa_t \psi = X_2(h,\psi) + 2 \sigma_0,
\end{equation}
where the unknowns $h,\psi$ are both functions, 
and the constant $2\sigma_0$ is determined by the fact that this is the only term 
for which the pair $(h,\psi) = (0,0)$ solves system \eqref{syst.pat.h.pat.psi.2.sigma.0}.
Note that $2 = H(0)$ is the mean curvature of the unit sphere $\S^2$, see \eqref{eq:meancurvature}.
Systems \eqref{syst.pat.h.pat.psi} and \eqref{syst.pat.h.pat.psi.2.sigma.0} are equivalent:
if $(h,\psi)$ solves \eqref{syst.pat.h.pat.psi.2.sigma.0}, 
then it also solves \eqref{syst.pat.h.pat.psi}; 
vice versa, if $(h, \psi)$ solves \eqref{syst.pat.h.pat.psi}, 
then there exists $\psi_1$ such that $\psi_1 \sim \psi$ and 
$(h, \psi_1)$ solves \eqref{syst.pat.h.pat.psi.2.sigma.0}.

\bigskip

Rotating traveling waves on the capillary drop 
are solutions of system \eqref{syst.pat.h.pat.psi.2.sigma.0} 
given by a fixed profile that rotates with constant angular velocity 
around a fixed symmetry axis. 
Given the absence of gravity in the equations, the problem has no privileged direction,
and hence there is no loss of generality in fixing the $x_3$ axis as the rotation symmetry axis. 
Thus, we consider functions $h, \psi$ of the form 
\begin{equation} \label{ansatz}
h(t,x) = \eta(R(\omega t)x), \quad \ 
\psi(t,x) = \beta(R(\omega t)x), \quad \ 
t \in \R, \ x \in \S^2,
\end{equation}
where $\eta, \beta : \S^2 \to \R$ are scalar functions defined on $\S^2$, independent of time, 
$\om \in \R$ is the angular velocity, 
and $R(\th)$, $\th \in \R$, 
is the rotation matrix 
\begin{equation}  \label{def.R(th)}
R(\th) = \begin{pmatrix} 
\cos \th & - \sin \th & 0 \\
\sin \th & \cos \th & 0 \\
0 & 0 & 1
\end{pmatrix}.
\end{equation}
As is observed in \cite{B.J.LM}, 
for functions $h,\psi$ of the form \eqref{ansatz}, 
system \eqref{syst.pat.h.pat.psi.2.sigma.0} becomes the rotating traveling wave system 
\begin{equation}
\label{syst.mM.eta.mM.beta}
\om \mM \eta = X_1(\eta, \beta), \quad \ 
\om \mM \beta = X_2(\eta, \beta) + 2 \sigma_0,
\end{equation}
where $\mM$ is the differential operator defined by 
\begin{equation}  \label{def.mM}
\mM \eta(x) := \la \mJ x , \grad_{\S^2} \eta(x) \ra \quad \ \forall x \in \S^2, 
\end{equation}
and $\mJ$ is the matrix 
\begin{equation}  \label{def.mJ}
\mJ := \begin{pmatrix} 
0 & - 1 & 0 \\
1 & 0 & 0 \\
0 & 0 & 0
\end{pmatrix}.
\end{equation}

Since $(\eta, \beta) = (0,0)$ is a solution of system \eqref{syst.mM.eta.mM.beta} 
for all values of $\om$, the existence of nontrivial solutions 
of \eqref{syst.mM.eta.mM.beta} can be studied as a bifurcation problem  
where the angular velocity $\om$ is the bifurcation parameter.
This study has been initiated in \cite{B.J.LM}, 
where the bifurcation from simple eigenvalues is proved by Crandall-Rabinowitz theorem. 
In the present paper we go beyond simple eigenvalues, and, 
following the approach of Craig-Nicholls \cite{Craig.Nicholls}, we prove that the bifurcation  
of nontrivial solutions of \eqref{syst.mM.eta.mM.beta} occur from eigenvalues of any multiplicity. 
In particular, we prove the following result. 
Let 
\begin{equation} \label{def.T.in.the.intro}
T := \{ (\ell,m) \in \Z^2 : \ell \geq 0, \ |m| \leq \ell \},
\end{equation}
the set of indices of the real spherical harmonics on $\S^2$. 
For $(\ell_0, m_0) \in T$, with $m_0 \neq 0$, define 
\begin{equation} \label{om.fix.in.the.intro}
\om_0 := \sqrt{\sigma_0} \frac{ \sqrt{(\ell_0+2)(\ell_0-1)\ell_0} }{m_0}
\end{equation}
and 
\begin{equation} \label{def.S.in.the.intro}
S := \{ (\ell, m) \in T :  (\ell+2)(\ell-1)\ell m_0^2 = (\ell_0+2)(\ell_0-1) \ell_0 m^2 \}.
\end{equation}
Note that the set $S$ has a finite number of elements, which are $(0,0)$, $(1,0)$, 
$(\ell_0, \pm m_0)$, and possibly finitely many other pairs $(\ell, \pm m)$ 
with $\ell \geq 2$ and $1 \leq |m| \leq \ell$.

\begin{theorem} \label{thm:main} 
Let $(\ell_0, m_0) \in T$, $\ell_0 \geq 2$, $m_0 \neq 0$,
where $T$ is in \eqref{def.T.in.the.intro}. 
Let $\om_0$ be defined in \eqref{om.fix.in.the.intro}. 
Let $2n+2$ be the cardinality of the set $S$ in \eqref{def.S.in.the.intro}.
Let $s > 1$.
Then there exist $a_0 > 0$, $C > 0$ such that, for every $a \in (0, a_0)$, 
there exist at least $n$ distinct orbits 
\begin{equation}  \label{n.orbits}
\{ (\eta_a^{(i)} , \beta_a^{(i)} ) \circ R(\th) : \th \in \T \}, \quad i = 1, \ldots, n, 
\end{equation}
of nontrivial solutions of the rotating traveling wave equations \eqref{syst.mM.eta.mM.beta}
with angular velocity $\om = \om_a^{(i)}$, 
angular momentum 
\begin{equation} \label{ang.mom.a}
\mI(\eta_a^{(i)} , \beta_a^{(i)}) = a
\end{equation}
($\mI$ is defined in \eqref{def.mI}), 
and 
\begin{equation} \label{est.sol}
|\om_a^{(i)} - \om_0| + \| \eta_a^{(i)} \|_{H^{s+\frac32}(\S^2)} + \| \beta_a^{(i)} \|_{H^{s+1}(\S^2)} 
\leq C \sqrt{a}.
\end{equation}
If $n=1$, the orbit depends analytically on $a$ in the interval $(0, a_0)$. 
\end{theorem}
Before continuing, let us explicitly write down an interesting consequence 
of Theorem \ref{thm:main} concerning the Cauchy problem for system \eqref{dyn.eq.01}-\eqref{kin.eq.01}.

\medskip

\textit{There exist infinitely many initial data
(i.e., the initial shape $\Omega_0$ and the initial velocity field $u_0$) 
arbitrarily close to the equilibrium (i.e., the sphere with null velocity) 
such that the only solution to the Cauchy problem for system \eqref{dyn.eq.01}-\eqref{kin.eq.01}
is obtained as a pure rotation of the initial data. 
In particular, these solutions do not develop singularities and exist for all times.}

 \medskip 

Let us briefly discuss why such an existence result is not really obvious at a heuristic level.
First, assume that we are given an irrotational solution 
of system \eqref{dyn.eq.01}-\eqref{kin.eq.01} 
with $\Om_t = \Om_0$ independent of time. 
Since the boundary $\pa \Om_t = \pa \Om_0$ does not change in time,
its normal velocity $V_t$ is zero for all $t$, 
and therefore, by \eqref{kin.eq.01}, $\la \grad \Phi , \nu_{\Om_0} \ra$ is also zero. 
By \eqref{div.eq.01}, $\Delta \Phi = \div u = 0$ in $\Om_0$, 
and, integrating the product $\Phi \Delta \Phi$ over $\Om_0$, 
using the divergence theorem and the boundary identity $\la \grad \Phi , \nu_{\Om_0} \ra = 0$, 
one obtains that $u = \grad \Phi = 0$ in $\Om_0$. 
Thus, by \eqref{dyn.eq.01}, $p$ is constant on the boundary $\pa \Om_0$, 
and hence, by \eqref{pressure.eq.01}, the curvature $H_{\Om_0}$ is also constant.
As a consequence, by Alexandrov theorem 
(a connected compact embedded surface with constant mean curvature is a sphere), 
$\Om_0$ is a ball. 
Thus, the ball is the only possible shape for an irrotational solution 
of \eqref{dyn.eq.01}-\eqref{kin.eq.01} with $\Om_t = \Om_0$.

Next, consider a solution of \eqref{dyn.eq.01}-\eqref{kin.eq.01}
where $\Om_t$ is given by the translation of a time-independent domain $\Om_0$ by constant velocity, 
i.e., $\Om_t = \Om_0 + ct$, for some fixed vector $c \in \R^3$. 
Then $\Om_t$, $u(t, x) = u_0(t, x-ct) + c$, $p(t,x) = p_0(t, x-ct)$ 
solve \eqref{dyn.eq.01}-\eqref{kin.eq.01}, 
where $\Om_0, u_0, p_0$ solve the same system. 
Hence, by the argument above, $\Om_0$ is a ball. 
Thus, the ball is also the only possible shape for an irrotational solution 
of \eqref{dyn.eq.01}-\eqref{kin.eq.01} with $\Om_t = \Om_0 + ct$ and $\Om_0$ independent of time.

From these observations one might be led to think 
that the same rigidity property that holds for the translations also holds for the rotations, 
namely, if the velocity field is irrotational, 
and if the domain $\Om_t$ is given by a rotation $R(\om t)$ 
of constant angular velocity applied to a time-independent domain $\Om_0$, 
then $\Om_0$ is a ball. 
Such a guess turns out to be false,
as is actually proved by Theorem \ref{thm:main}.
In particular, the solutions obtained in Theorem \ref{thm:main} have domain $\Om_t = R(\om t) \Om_0$, 
where $\Om_0$ is independent of time and it is not a ball.
This is related to the fact that system \eqref{syst.pat.h.pat.psi.2.sigma.0} 
concernes the time evolution not only of the shape $h$, but also of the velocity potential $\psi$, 
which can be suitably tuned so that a fixed nonspherical shape persists as time evolves.

Another, even less intuitive, consequence of Theorem \ref{thm:main} is that, in general, 
rotating solutions need not to enjoy any symmetry property to exist.

\medskip

Theorem \ref{thm:main} improves the bifurcation result in \cite{B.J.LM} 
in these aspects: 
\begin{itemize}
\item[-] in \cite{B.J.LM} bifurcation is proved only 
from frequencies $(\ell_0, m_0) \in T$ such that 
the set $S$ in \eqref{def.S.in.the.intro} has exactly the 4 elements 
$(0,0), (1,0), (\ell_0, \pm m_0)$, 
whereas in Theorem \ref{thm:main} there is no restriction on the cardinality of $S$;

\item[-] in \cite{B.J.LM} bifurcation is proved only in the subspace of functions 
$(\eta, \beta)$ with $\eta$ even in $x_3$, even in $x_2$, 
and $\beta$ even in $x_3$, odd in $x_2$, 
where the frequency $(\ell_0, m_0)$ becomes simple,
whereas in Theorem \ref{thm:main} there is no restriction on the symmetry properties of $(\eta, \beta)$; 

\item[-] in \cite{B.J.LM} a bifurcating curve of nontrivial solutions is obtained, 
whereas Theorem \ref{thm:main} gives a multiplicity result, 
proving the bifurcation of $n$ orbits of nontrivial solutions.
\end{itemize}

We also note that the analytic dependence on the bifurcation parameter proved in \cite{B.J.LM}
for simple eigenvalues is lost in Theorem \ref{thm:main}, 
as one expects in the case of multiple eigenvalues; 
for $n=1$, on the other hand, the dependence is analytic also in Theorem \ref{thm:main}.  

\medskip

The multiplicity result in Theorem \ref{thm:main} relies on the torus action 
given by the rotations $R(\th)$, see \eqref{n.orbits}, 
and it does not make use of the reversible structure of the problem. 
Reversibility would enter into the result in the following way.
Given any solution $u = (\eta, \beta)$ of equation \eqref{syst.mM.eta.mM.beta} 
with angular velocity $\om > 0$,  
the pair 
\[
\mS u := (\eta, - \beta)
\]
is a solution of the same equation  
but with $\om$ replaced by its opposite $- \om < 0$. 
Since the $n$ orbits in \eqref{n.orbits} are all made by solutions 
of \eqref{syst.mM.eta.mM.beta} with positive angular velocity 
($\om_a^{(i)}$ is close to $\om_0 > 0$, see \eqref{est.sol}), 
no one of them can be obtained from another one by the reflection $\mS$, 
i.e., 
\[
(\eta^{(i)}, \beta^{(i)}) \neq \mS \{ (\eta^{(j)}, \beta^{(j)}) \circ R(\th) \} 
\quad \ \forall \th \in \T, \ \ i \neq j. 
\]
Thus, taking into account the involution $\mS$, we obtain $2n$ orbits of rotating traveling waves, 
$n$ of which having anti-clockwise rotation, 
and the other $n$ with clockwise rotation. 
The same happens for $\mS$ replaced by the involution operator $\mS_1$ 
that maps $u = (\eta, \beta)$ into 
\[
(\mS_1 u)(x) := (\eta(-x), - \beta(-x)), 
\]
which is the composition of $\mS$ with the map $x \to -x$, 
considered, e.g., in \cite{BBMM}.
This change of sign of the angular velocity is due to the fact that the angular momentum $\mI$ 
defined in \eqref{def.mI} satisfies 
$\mI \circ \mS = - \mI$ and $\mI \circ \mS_1 = - \mI$, 
unlike the linear momentum considered in the nearly flat case, 
for instance in \cite{Craig.Nicholls, BBMM}.
This difference is ultimately due to the fact that the space derivative operator 
in the momentum integral is $\pa_x$ or $\grad_x$ in the nearly flat case, 
while it is $x_1 \pa_{x_2} - x_2 \pa_{x_1}$ in the nearly spherical one. 

\medskip

Concerning symmetric solutions, we prove the following result. 

\begin{theorem} \label{thm:symmetric}
$(i)$ 
\emph{(Solutions even in $x_3$)}.
Let $(\ell_0, m_0) \in T$, $\ell_0 \geq 2$, $m_0 \neq 0$, 
with $\ell_0 - m_0$ an even integer,
where $T$ is in \eqref{def.T.in.the.intro}. 
Let $\om_0$ be defined in \eqref{om.fix.in.the.intro}. 
Let $2n+1$ be the cardinality of the set $S \cap \Z^2_{even}$, 
where $S$ is defined in \eqref{def.S.in.the.intro}, 
and $\Z^2_{even}$ is the set of pairs $(\ell,m) \in \Z^2$ 
such that $\ell - m$ is even.
Let $s > 1$.
Then there exist $a_0 > 0$, $C > 0$ such that, for every $a \in (0, a_0)$, 
there exist at least $n$ distinct orbits \eqref{n.orbits}
of solutions of the rotating traveling wave equations \eqref{syst.mM.eta.mM.beta}
with symmetry with respect to the third variable $x_3$
\begin{equation} \label{symm.3}
( \eta_a^{(i)} , \beta_a^{(i)} )(x_1, x_2, - x_3) 
= ( \eta_a^{(i)} , \beta_a^{(i)} )(x_1, x_2, x_3),
\quad \ x = (x_1, x_2, x_3) \in \S^2,
\end{equation}
angular velocity $\om = \om_a^{(i)}$, 
angular momentum \eqref{ang.mom.a}, 
and estimates \eqref{est.sol}.
If $n=1$, the orbit depends analytically on $a$ in the interval $(0, a_0)$. 

\medskip

$(ii)$ 
\emph{(Solutions even/odd in $x_2$)}.
Let $(\ell_0, m_0) \in T$, $\ell_0 \geq 2$, $m_0 \neq 0$. 
Let $\om_0$ be defined in \eqref{om.fix.in.the.intro}. 
Let $s > 1$.
Then there exist $a_0 > 0$, $C > 0$ such that, for every $a \in (0, a_0)$, 
there exist at least $2$ solutions 
\begin{equation}  \label{2.sol}
(\eta_a^{(i)} , \beta_a^{(i)} ), \quad i = 1,2,
\end{equation}
of \eqref{syst.mM.eta.mM.beta} 
with symmetry with respect to the second variable $x_2$ 
\begin{equation} \label{symm.2}
\eta_a^{(i)} (x_1, - x_2, x_3) 
= \eta_a^{(i)} (x_1, x_2, x_3), 
\quad \ 
\beta_a^{(i)} (x_1, - x_2, x_3)
= - \beta_a^{(i)} (x_1, x_2, x_3),
\end{equation}
with angular velocity $\om = \om_a^{(i)}$, 
angular momentum \eqref{ang.mom.a}, 
and estimates \eqref{est.sol}.

\medskip

$(iii)$ 
\emph{(Solutions even in $x_3$ and even/odd in $x_2$)}.
Let $(\ell_0, m_0) \in T$, $\ell_0 \geq 2$, $m_0 \neq 0$, 
with $\ell_0 - m_0$ even. 
Let $\om_0$ be defined in \eqref{om.fix.in.the.intro}. 
Let $s > 1$.
Then there exist $a_0 > 0$, $C > 0$ such that, for every $a \in (0, a_0)$, 
there exist at least $2$ solutions \eqref{2.sol}
of \eqref{syst.mM.eta.mM.beta} 
with both symmetries \eqref{symm.3} and \eqref{symm.2}, 
angular velocity $\om = \om_a^{(i)}$, 
angular momentum \eqref{ang.mom.a} and estimates \eqref{est.sol}.
\end{theorem}

Theorem \ref{thm:symmetric} improves the bifurcation result in 
\cite{B.J.LM}, because in \cite{B.J.LM} only the case \textit{(iii)} is considered, 
and only when $S$ has exactly 4 elements.  

We observe that from the torus action one deduces the invariance 
of the subspaces of functions $(\eta, \beta)$ which are supported only on spherical harmonics 
$\ph_{\ell,m}$ with $m$ integer multiple of a fixed integer $k > 1$. 
This is usually called $k$-fold symmetry, and it corresponds 
to the discrete rotational symmetry of angle $2 \pi / k$ around the $x_3$ axis. 
Results analogous to those of Theorems \ref{thm:main} and \ref{thm:symmetric}
can be also proved in the $k$-fold symmetry subspaces.

\bigskip 

\emph{Sketch of the proof.}
A fundamental ingredient in Craig-Nicholls approach \cite{Craig.Nicholls} 
to the bifurcation of traveling waves is the variational structure of the problem, 
namely the key observation that the traveling wave system is the equation of the critical points 
of a functional; thus, our analysis begins with showing 
that system \eqref{syst.mM.eta.mM.beta}
exhibits a variational structure analogous to the one in \cite{Craig.Nicholls}. 
More precisely, in Section \ref{sec:variational.formulation} we recall the Hamiltonian structure of the problem,
and we show that \eqref{syst.mM.eta.mM.beta} 
is in fact the equation 
\begin{equation} \label{eq.critical.point.in.the.intro}
\grad \mH_{\sigma_0}(\eta, \beta) 
- \om \grad \mI(\eta, \beta) = 0
\end{equation}
of the critical points of the functional $\mH_{\sigma_0} - \om \mI$,
where $\mH_{\sigma_0}$ is the Hamiltonian defined in \eqref{def.mH.sigma.0}
and $\mI$ is the angular momentum defined in \eqref{def.mI}.

Next, we show some consequences of the geometric nature of the problem: 
in Section \ref{sec:symmetries} we use some geometric considerations to 
find symmetries of the problem
and we follow the ideas of Noether Theorem to find the corresponding conserved quantities. 
An important tool that we also study in Section \ref{sec:symmetries} is the operator $\mT_{\theta}$ 
in \eqref{def.mT.theta}, which is the natural torus action of the problem, 
for which we prove some key invariance properties, see Lemma \ref{lemma:group.action}. 
At this point, we are ready to start the study of equation \eqref{eq.critical.point.in.the.intro}. 

In Section \ref{sec:nonlinear} we 
prove some orthogonality (Lemma \ref{lemma:orthogonality.mF}) 
and conjugation (Lemma \ref{lemma:conj.mF}) properties 
of the nonlinear operator in \eqref{eq.critical.point.in.the.intro}.
These two simple lemmas will have important consequences: 
orthogonality will be used at the end of the proof of Theorem \ref{thm:main}, 
see Lemma \ref{lemma:sol.bif}, and it will be essential 
in order to show the existence of at least two solutions 
to \eqref{eq.critical.point.in.the.intro}, 
while the conjugation property will be used to obtain $n$ orbits of solutions, 
see Lemma \ref{lemma:n.critical.orbits}.

In Section \ref{sec:linearized} we linearize the system at the origin 
and list some properties of the linearized operator, 
building the appropriate set up to implement the Lyapunov-Schmidt decomposition 
in Section \ref{sec:LS.decomposition}. 
The original problem is then reduced to a finite-dimensional one
by solving the range equation by the implicit function theorem, see Lemma \ref{lemma:range.eq}. 
It is important to observe that the reduced problem is now an equation in even dimensional space; 
more precisely, we are left to solve the bifurcation equation in a vector space of dimension $2n+2$. 

Following \cite{Craig.Nicholls} and the classical methods described in 
\cite{Chow.Hale, Mahwin.Willem, Struwe},  
we note that, even after performing the Lyapunov Schmidt decomposition, 
the bifurcation equation 
has still a variational structure, see Lemma \ref{lemma:variational.nature.ZN}. 
Hence, to solve the bifurcation equation, we look for critical points 
of the reduced functional $\mE_{V_N}$. 
To this purpose, we introduce the level set of the angular momentum  
as a constraint where to look for constrained critical points.

A close inspection of the constraint $\mI(u)=a$ makes clear the necessity of a further reduction, 
as the level sets of the angular momentum resemble a cylinder of the type 
$\mathbb{R}^2 \times \mathbb{S}^{2n-1}$ (see Lemma \ref{IZ.formula}).  
Hence, in Section \ref{sec:degenerate.decomposition} 
we further decompose the problem separating the two \emph{degenerate} directions in $\R^2$
from the other $2n$ \emph{nondegenerate} directions on $\S^{2n-1}$. 
Note that the formula of the main order $\mI_0(v)$ in the Taylor expansion 
of the reduced functional $\mI_{V_N}(v)$ in Lemma \ref{IZ.formula} 
does not depend on the degenerate variables.

Then, in Section \ref{sec:choice.of.omega}, 
we choose the angular velocity parameter $\om$ in an efficient way
as a function of the nondegenerate component of the kernel 
(Lemma \ref{lemma:choice.of.omega}). 
This choice of $\om$ guarantees that constrained critical points of the reduced functional 
are, in fact, critical points, and therefore solutions of the nondegenerate bifurcation equation
(Lemma \ref{lemma:sol.bif.ZN}). Thus the level sets of the reduced angular momentum functional 
$\mI_{V_N}(v) = a$ become \emph{natural constraints} (in the sense of \cite{Ambrosetti.Malchiodi})
for the reduced functional $\mE_{V_N}(v)$. 
Actually, we have that critical points of the restricted functional 
(defined on a $2n$ dimensional space) 
invariant under the group action $\mT_\theta$ on a manifold 
diffeomorphic to $\S^{2n-1}$ 
are solutions to the bifurcation equation. 
Hence, by using classical theorems of critical point theory 
(Benci $\S^1$-index), we infer the existence of $n$ distinct orbits of solutions, 
see Lemma \ref{lemma:n.critical.orbits}.

Finally, in Section \ref{sec:sol.bif.deg}, we solve the bifurcation equation 
also in the two degenerate directions, thanks to the orthogonality property 
of Lemma \ref{lemma:orthogonality.mF}.

\medskip

\emph{Difficulties in the proof.} 
Although the strategy of the proof is well established, 
we encounter some nontrivial challenges in the way to prove Theorem \ref{thm:main}, 
which seem to be peculiar of the problem of the drop with its spherical geometry, 
in the sense that they are not present for the more studied case 
of traveling waves on a flat reference manifold (the ocean). 
We try to explain this point.

In the problem of traveling gravity-capillary waves on a flat 2-dimensional lattice 
considered by Craig and Nicholls \cite{Craig.Nicholls}, 
there is one degenerate direction, which corresponds to the zeroth Fourier coefficient 
$\hat \psi_0$ of the velocity potential unknown $\psi$. 
This degeneration is easily removed in \cite{Craig.Nicholls} 
by restricting all the analysis to the \emph{linear} invariant subspace of functions 
$(h, \psi)$ where $\psi$ has zero average over the lattice, i.e., $\hat \psi_0 = 0$.  
Also, the conservation of the fluid mass in \cite{Craig.Nicholls} 
corresponds to the \emph{linear} identity $\hat h_0 = $ constant. 
The same happens, for example, also in the recent paper \cite{BBMM} with constant vorticity. 
In fact, this can also be achieved for the drop: 
the first orthogonality relation in \eqref{eq:orthogonal.mV} 
(that is, the first orthogonality relation in \eqref{eq:orthogonal.vectors}) 
gives an invariant \emph{linear} subspace, 
while the mass conservation, despite its cubic expression \eqref{def.volume} 
in the elevation function $h$, 
could become the \emph{linear} identity $\hat h_0 = $ constant 
by means of a simple change of variable, as explained in Lemma 3.4 in \cite{B.J.LM}. 

However, the problem for the drop contains \emph{another} degenerate direction, 
which corresponds to the conservation of the $x_3$-component $\mB_3$ 
of the barycenter velocity $\mB$, see \eqref{def.mB}.
To better explain, let us describe the simple action of moving the barycenter of the drop 
along the $x_3$ axis. 
Geometrically and physically, this is probably the simplest operation one can perform 
and yet the reparametrization of a translated drop is an implicitly defined nonlinear diffeomorphism. 
From a physical point of view, it is clear that if a drop rotating around the 
$x_3$ axis is a solution to \eqref{syst.pat.h.pat.psi}, 
any other drop that rotates in the same fashion and translates along $x_3$ 
with constant velocity is still a solution. 

Now, unlike the previous degeneration, the conservation of $\mB_3$ 
is not linear, and it is hard, if not even impossible, 
to find a change of variable that simultaneously turns both the conservation relations 
into linear ones. 
This would force one to work over infinitely-dimensional invariant \emph{manifolds} 
of codimension 2. 
This technical issue is not present in the flat case. 
At a linear level, this is related to the fact that the Laplace-Beltrami operator 
$- \Delta_{\S^2}$ has eigenvalues $\ell (\ell - 1)$, 
which vanishes for $\ell = 0$ and also for $\ell = 1$, 
in contrast with $- \Delta_{\T^2}$. 

From a technical point of view, we encounter the degeneracy 
in Section \ref{sec:degenerate.decomposition}, 
where, after performing the Lyapunov-Schmidt decomposition, 
following the method of \cite{Chow.Hale, Craig.Nicholls, Mahwin.Willem, Struwe}, 
we have to fix the bifurcation parameter $\om$ as a function 
of the kernel component $v$ by implicit function theorem, 
and we face the fact that the dominant term \eqref{approx.I.Z.norm}
depends nontrivially on all the components of $v$ but two, 
$\hat v_{0,0}$ and $\hat v_{1,0}$, which are the degenerate directions.

In fact, we choose not to perform any change of variable 
to make the mass conservation a linear expression of $h$, 
and we also choose not to work on infinitely-dimensional invariant manifolds. 
Instead, the strategy that we find more convenient to overcome this obstacle 
is to exploit the two orthogonality conditions 
\eqref{eq:orthogonal.mV} (that is, \eqref{eq:orthogonal.vectors}) 
coming from the Hamiltonian structure of our problem and Noether Theorem.
While the first relation in \eqref{eq:orthogonal.vectors} is a simple linear expression, 
the second one is not, but it gives a geometrical \emph{transversality condition} 
in the vicinity of the bifurcation point. 
This transversality property is sufficient to cancel out the Lagrange multipliers 
related to the two degenerate directions in the proof of Lemma \ref{lemma:sol.bif}.

Another nontrivial aspect of working on $\S^2$ 
concerns the compatibility of the Lyapunov-Schmidt decomposition 
with the torus action $\mT_\th$, which does not seem to be a priori obvious, 
because of the nontrivial algebraic expression of the spherical harmonics. 
This compatibility is proved in Lemma \ref{lemma:V.W.group.action}.

\bigskip

\emph{Related literature about the drop.} 
The problem of the fluid motion of a capillary nearly spherical drop 
dates back to the paper \cite{Lord.Rayleigh.jets} by Lord Rayleigh. 
The formulation of the free boundary problem for the drop in the irrotational regime 
as a system of equations on $\S^2$ 
has been used in \cite{B.J.LM, Beyer.Gunther, Shao.initial.notes, Shao.G.paralinearization}. 
Its Hamiltonian structure has been obtained in \cite{Beyer.Gunther} 
and, in the present formulation, in \cite{B.J.LM}. 
The Dirichlet-Neumann operator for the drop 
has been studied in \cite{Shao.G.paralinearization} (paralinearization)
and \cite{B.J.LM} (linearization, analyticity, tame estimates). 
Local well-posedness results for the Cauchy problem have been obtained in 
\cite{Beyer.Gunther, CS, Shao.G.paralinearization}, 
and continuation results and a priori estimates in \cite{JL, Shatah.Zeng}.
The existence of rotating traveling waves has been analytically proved, 
under symmetry assumptions, in \cite{B.J.LM} and, in dimension 2, in \cite{Moon.Wu}, 
while its numerical evidence is shown in \cite{Dyachenko}; 
we mention also \cite{MNS} for the existence of 2D steady bubbles for capillary drops and \cite{MRS} for the existence of 2D steady vortex sheets.
 
\medskip

\emph{Related literature about the bifurcation and critical point theory.}
Concerning the bifurcation technique, the present paper is inspired by \cite{Craig.Nicholls}, 
where the existence of capillary-gravity traveling waves on 2D and 3D flat torus is proved, 
relying on the proof of the Lyapunov Centre Theorem in \cite{Weinstein} and \cite{Moser}; 
a similar approach can also be found in the very recent paper \cite{BBMM}. 
We refer to \cite{Benci, Chow.Hale, Mahwin.Willem, Struwe} as far as the Benci index 
and its applications in proving multiplicity of critical points are concerned. 
We refer to \cite{Ambrosetti.Malchiodi, Ambrosetti.Prodi, Ambrosetti.Rabinowitz, 
Buffoni.Toland, Chow.Hale, Crandall.Rabinowitz, Fadell.Rabinowitz, Mahwin.Willem, Struwe}
for further literature about bifurcation and critical point theories,  
and to \cite{Berti.Bolle, Berti.Bolle.1, BLS} for some interesting applications to PDEs. 

\medskip

\emph{Related literature about traveling water waves.} 
Starting from the pioneering works \cite{Stokes}, \cite{Struik}, \cite{Levi.Civita} and \cite{Nekrasov},  
the literature about the water wave problem in the flat (i.e., non spherical) case, 
and in particular on traveling waves, is huge. 
For a recent review, we refer to \cite{HHSTWWW}.
Concerning the Hamiltonian structure and the Dirichlet-Neumann operator, 
important references are \cite{AM, CSS, Craig.Sulem, Lannes, Zakharov}; 
see also \cite{BMV, Wahlen.1} for recent developments.
For the Cauchy problem we refer, for instance, to \cite{ABZ, AD, Berti.Delort, Lannes}. 
For traveling waves we mention, e.g., 
\cite{Groves.Sun, Iooss.Plotnikov, Iooss.Plotnikov.1} for 3D problems,
\cite{ASW} for global bifurcation, 
\cite{GNPW} for Beltrami flows, 
\cite{Wahlen} for nonzero vorticity,
\cite{Jones.Toland} for secondary bifurcations, 
\cite{BFM, BFMT, CMT, Feola.Giuliani} for quasi-periodic traveling waves.

\bigskip

\noindent
\textbf{Acknowledgments}. 
The authors warmly thank Vittorio Coti Zelati and Massimiliano Berti 
for interesting discussions and suggestions.
This work is supported by Italian INdAM group GNAMPA, 
by Italian Project PRIN 2022E9CF89 \emph{Geometric Evolution Problems and Shape Optimization}, 
and by Italian Project PRIN 2020XB3EFL \emph{Hamiltonian and dispersive PDEs}.

\section{Notations and preliminaries} 
\label{sec:prelim}

\underline{Tangential gradient}, \underline{tangential Laplacian}. 
Given a function $f: \S^2\to \R$, 
we denote by $\mE_0 f$ its 0-homogenous extension to $\R^3 \setminus \{ 0 \}$
\[
\mE_0 f(x) = f \Big( \frac{x}{|x|} \Big), \quad \ x \in \R^3 \setminus \{ 0 \}.
\]  
Using the extension operator $\mE_0$, we define 
the tangential gradient $\nabla_{\S^2} f$ 
and the tangential Laplacian $\Delta_{\S^2} f$  
of a function $f : \S^2 \to \R$ as 
the classical gradient and Laplacian of $\mE_0 f$,
\begin{equation} \label{def.grad.S2.Lap.S2}
\nabla_{\S^2} f = \nabla \mE_0 f, 
\quad \  
\Delta_{\S^2} f = \Delta \mE_0 f. 
\end{equation}
By using the definition of tangential gradient in \eqref{def.grad.S2.Lap.S2}, 
it also holds that, if $\tilde f$ is any diffentiable extension of $f$ in a neighborhood of $\S^2$, then
\begin{equation} \label{tang.grad.any.ext}
\nabla_{\S^2} f = \nabla \tilde f - \langle \nabla \tilde f, x\rangle x.
\end{equation}

\noindent
\underline{Mean curvature}. 
The mean curvature term $H(h)$ in equation \eqref{def.X.2} is given by
\begin{equation} \label{eq:meancurvature}
H(h) 
= - \frac{\Delta_{\S^2} h}{(1+h)J} 
+ \frac{2}{J} 
+ \frac{ \la (D_{\S^2}^2 h) \nabla_{\S^2} h, \nabla_{\S^2} h \ra}{(1+h) J^3} 
+ \frac{ |\nabla_{\S^2} h|^2}{J^3},
\end{equation}
where 
\begin{equation}  \label{def.J}
J = \sqrt{(1+h)^2+ |\nabla_{\S^2}h |^2},
\quad \ 
\la (D_{\S^2}^2 h) \nabla_{\S^2} h, \nabla_{\S^2} h \ra 
= \la (D^2 \mE_0 h) \grad \mE_0 h , \grad \mE_0 h \ra,
\end{equation}
and $D^2 \mE_0 h$ is the Hessian matrix of $\mE_0 h$. 
Formula \eqref{eq:meancurvature} is proved in Lemma 2.3 of \cite{B.J.LM}.

\medskip

\noindent
\underline{Divergence Theorem on $\S^2$}.
Given a vector field $F$ on $\S^2$, one has  
\begin{equation} \label{div.thm.S2} 
\int_{\S^2} \div_{\S^2} F \, d\sigma 
= 2 \int_{\S^2} \la F, x \ra \, d\sigma,
\end{equation}
where $d \sigma$ is the 2-dimensional Hausdorff measure, and 
the tangential divergence is 
\[
\div_{\S^2} F = \div \mE_0 F.
\]

\noindent
\underline{Spherical harmonics}. 
The real Hilbert space $L^2(\S^2, \R)$ admits orthonormal Hilbert basis 
made of real spherical harmonics. 
In particular, we consider the $L^2(\S^2, \R)$ orthonormal Hilbert basis 
\begin{equation}  \label{def.ph.ell.m.intro}
\{ \ph_{\ell,m} : (\ell, m) \in T \}, \quad \
T := \{ (\ell, m) \in \Z^2 : \ell \geq 0, \ m = - \ell, \ldots, \ell \},
\end{equation}
where $\ph_{\ell,m}(x)$, $m = - \ell, \ldots, \ell$, 
are the Legendre real spherical harmonics 
on $\S^2$ of degree $\ell$ written in Cartesian coordinates $x \in \S^2 \subset \R^3$, 
namely 
\begin{align}
\ph_{\ell,0}(x) & := c_{\ell}^{(0)} P_\ell(x_3), 
\notag \\ 
\ph_{\ell,m}(x) & 
:= c_{\ell}^{(m)} P_\ell^{(m)}(x_3) \Re[ (x_1 + i x_2)^m ], 
\quad m = 1, \ldots, \ell,
\notag \\
\ph_{\ell,-m}(x) & 
:= c_{\ell}^{(m)} P_\ell^{(m)}(x_3) \Im[ (x_1 + i x_2)^{m} ], 
\quad m = 1, \ldots, \ell, 
\label{def.ph.ell.m.Re.Im}
\end{align}
where $P_\ell$ is the ordinary Legendre polynomial of degree $\ell$, 
$P_\ell^{(m)}$ is its $m$th derivative, 
and $c_\ell^{(m)} \in \R$ is a normalizing coefficient; 
see, e.g., \cite{Atkinson.Han}, Example 2.48 in Section 2.11, 
and \cite{B.J.LM}, Section 6.1. 
For all $(\ell,m) \in T$, one has
$$
-\Delta_{\S^2} \ph_{\ell,m} = \ell(\ell+1) \ph_{\ell,m}.
$$

\medskip

\noindent
\underline{Sobolev spaces}. 
Given any function $u\in L^2(\S^2,\R)$, one has 
\[
u= \sum_{(\ell,m) \in T} \hat u_{\ell,m} \ph_{\ell,m}, \quad \ 
\hat u_{\ell,m} 
= \la u , \ph_{\ell,m} \ra_{L^2(\S^2)}
= \int_{\S^2} u \ph_{\ell,m} \, d \sigma. 
\]
For any real $s \geq 0$, we define the Sobolev space $H^s(\S^2,\R)$ as
\begin{equation}
H^s(\S^2,\R) = \{u \in L^2(\S^2,\R) : \| u \|_{H^s(\S^2)} < \infty \}, 
\quad \ 
\| u \|_{H^s(\S^2)}^2 
= \hat u^2_{0,0} + \sum_{\begin{subarray}{c} (\ell,m) \in T \\ \ell \geq 1 \end{subarray}} 
\ell^{2s} \hat u^2_{\ell,m}.
\end{equation}
Sometimes, we shortly write $L^2(\S^2), H^s(\S^2)$ 
instead of $L^2(\S^2,\R), H^s(\S^2,\R)$.

\section{Variational formulation of the problem} 
\label{sec:variational.formulation}

As is proved in \cite{Beyer.Gunther} and \cite{B.J.LM}, 
system \eqref{syst.pat.h.pat.psi.2.sigma.0} has a Hamiltonian structure. 
Consider the quantity 
\begin{align} 
\mathcal{H}(h,\psi) 
& := \frac12 \int_{\mathbb{S}^2} \psi (G(h)\psi) (1+h) \sqrt{(1+h)^2 + |\nabla_{\mathbb{S}^2} h|^2} \, d\sigma
\notag \\ & \quad \ 
+ \sigma_0 \int_{\mathbb{S}^2} (1+h) \sqrt{(1+h)^2 + |\nabla_{\mathbb{S}^2} h|^2} \, d\sigma,
\label{def.mH}
\end{align}
which is the sum of the kinetic energy and the capillarity energy of the drop 
expressed in terms of the functions $h, \psi$.  
Then (see Proposition 3.2 in \cite{B.J.LM}) 
\begin{equation}  \label{quasi.Ham.syst}
X_1(h, \psi) 
= \frac{ \pa_\psi \mathcal H(h,\psi) }{ (1+h)^2 },
\quad \ 
X_2(h, \psi) 
= - \frac{ \pa_h \mathcal H(h,\psi) }{ (1+h)^2 },
\end{equation}
where $X_1, X_2$ are defined in \eqref{def.X.1}, \eqref{def.X.2} 
and $\pa_h \mH$, $\pa_\psi \mH$ are the gradients of $\mH$ 
with respect to the $L^2(\S^2)$ scalar product. 
The factors $(1+h)^{-2}$ in \eqref{quasi.Ham.syst} can be easily removed with a change of variables 
to obtain a Hamiltonian system in Darboux coordinates, 
see Lemma 3.4 in \cite{B.J.LM}; 
however, this is not necessary for the analysis of the present paper, 
and therefore we do not do it. 
We also consider the quantity
\begin{equation}  \label{def.volume}
\mV(h,\psi) := \mV(h) := \frac13 \int_{\S^2} (1+h)^3 \, d\sigma, 
\end{equation}
which is the volume of the fluid domain $\Om_t$ written in terms of the elevation function $h$.  
Its gradients with respect to the $L^2(\S^2)$ scalar product are 
\begin{equation} \label{pa.mV}
\pa_h \mV(h,\psi) = (1+h)^2, \quad \ 
\pa_\psi \mV(h,\psi) = 0.
\end{equation}
Hence, by \eqref{quasi.Ham.syst} and \eqref{pa.mV}, one has 
\begin{equation}  \label{quasi.Ham.syst.2.sigma.0}
X_1(h, \psi) = \frac{ \pa_\psi \mathcal H_{\sigma_0}(h,\psi) }{ (1+h)^2 },
\quad \ 
X_2(h, \psi) + 2 \sigma_0 
= - \frac{ \pa_h \mathcal H_{\sigma_0}(h,\psi) }{ (1+h)^2 },
\end{equation}
where 
\begin{equation} \label{def.mH.sigma.0}
\mH_{\sigma_0} := \mH - 2 \sigma_0 \mV,
\end{equation}
and \eqref{syst.pat.h.pat.psi.2.sigma.0} is the Hamiltonian system 
\begin{equation}  \label{quasi.Ham.syst.2.sigma.0.bis}
\pa_t h = \frac{ \pa_\psi \mathcal H_{\sigma_0}(h,\psi) }{ (1+h)^2 },
\quad \ 
\pa_t \psi = - \frac{ \pa_h \mathcal H_{\sigma_0}(h,\psi) }{ (1+h)^2 }.
\end{equation}
Also, by \eqref{quasi.Ham.syst.2.sigma.0}, 
the rotating traveling wave system \eqref{syst.mM.eta.mM.beta} can be written as 
\begin{equation}  \label{quasi.Ham.syst.eta.beta}
\om \mM \eta = \frac{ \pa_\beta \mathcal H_{\sigma_0}(\eta,\beta) }{ (1+\eta)^2 },
\quad \ 
\om \mM \beta = - \frac{ \pa_\eta \mathcal H_{\sigma_0}(\eta,\beta) }{ (1+\eta)^2 }.
\end{equation}

Now we show that also the terms $\mM \eta, \mM \beta$ 
are the product of the factor $(1 + \eta)^{-2}$ with the gradients 
of a functional with respect to the $L^2(\S^2)$ scalar product, 
and this functional is the angular momentum $\mI$ with respect to the rotation symmetry axis.
Before proving this, we show in the following lemma
that the differential operator $\mM$ is anti-symmetric, 
i.e., one has a simple integration by parts formula for $\mM$,
and we calculate its commutator with the tangential gradient operator.

\begin{lemma} 
\label{lem:commutation.formula}
$(i)$ One has 
\begin{equation}  \label{mM.is.anti.symm}
\int_{\S^2} f \mM g \, d\sigma = - \int_{\S^2} (\mM f) g \, d\sigma
\end{equation}
for all $f, g \in C^1(\S^2, \R)$.

$(ii)$ One has 
\begin{equation} \label{mM.explicit}
\mM f(x) 
= \la \mJ x , \grad \tilde f(x) \ra 
= x_1 \pa_{x_2} \tilde f(x) - x_2 \pa_{x_1} \tilde f(x) 
\end{equation} 
for all $x \in \S^2$, all $f \in C^1(\S^2, \R)$, 
where $\tilde f$ is any extension of $f$ to a neighborhood of $\S^2$; 
in particular, \eqref{mM.explicit} holds even if $\tilde f$ is not a homogeneous function. 

$(iii)$ One has the commutator identity
\begin{equation}  \label{mM.grad.commutator}
\mM (\grad_{\S^2} f) = \grad_{\S^2} (\mM f) + \mJ \grad_{\S^2} f
\end{equation} 
for all $f \in C^2(\S^2, \R)$. 
\end{lemma}

\begin{proof}
$(i)$ Consider the vector field $F(x) = f(x) g(x) \mJ x$ on $\S^2$, 
and consider its 0-homogeneous extension $\mE_0 F = (\mE_0 f) (\mE_0 g) \mJ x |x|^{-1}$.
Since the matrix $\mJ$ in \eqref{def.mJ} satisfies 
\[
\div (\mJ x) = 0, \quad \ 
\la \mJ x , x \ra = 0, \quad \ 
\div(\mJ x |x|^{-1}) = 0 
\]
for all $x \in \R^3$, $x \neq 0$, 
the tangential divergence of $F$ on $\S^2$ is 
\[
\div_{\S^2} F = \div (\mE_0 F) 
= \la \grad \{ (\mE_0 f)(\mE_0 g) \} , \mJ x \ra 
= (\mM f) g + f \mM g,
\]
and the divergence theorem on $\S^2$ in \eqref{div.thm.S2} 
gives identity \eqref{mM.is.anti.symm}.

$(ii)$ By \eqref{tang.grad.any.ext}, since $\la \mJ x , x \ra = 0$, one has 
\[
\mM f 
= \la \mJ x , \grad \tilde f - \la \grad \tilde f , x \ra x \ra 
= \la \mJ x , \grad \tilde f \ra.
\]

$(iii)$ 
One has $\grad_{\S^2} f = \grad \mE_0 f$ on $\S^2$, 
and $\grad \mE_0 f$ is defined in $\R^3 \setminus \{ 0 \}$. 
Hence $\grad \mE_0 f$ is an extension of $\grad_{\S^2} f$, 
and, by \eqref{mM.explicit}, 
\begin{equation} \label{mM.grad.temp.01}
\mM ( \grad_{\S^2} f ) = (x_1 \pa_{x_2} - x_2 \pa_{x_1}) \grad \mE_0 f.
\end{equation}
The scalar function $g := (x_1 \pa_{x_2} - x_2 \pa_{x_1}) \mE_0 f$  
is defined in $\R^3 \setminus \{ 0 \}$, 
it is a 0-homogeneous function, and it coincides with $\mM f$ on $\S^2$. 
Hence $g$ is the 0-homogeneous extension $\mE_0 (\mM f)$ of the function $\mM f$, 
and, by definition \eqref{def.grad.S2.Lap.S2}, one has 
\[
\grad_{\S^2} (\mM f) 
= \grad \mE_0 (\mM f) 
= \grad g 
= \grad \{ (x_1 \pa_{x_2} - x_2 \pa_{x_1}) \mE_0 f \}
= (x_1 \pa_{x_2} - x_2 \pa_{x_1}) \grad \mE_0 f - \mJ \grad \mE_0 f,
\]
where the last identity is an elementary calculation. 
Thus, recalling \eqref{mM.grad.temp.01}, we have proved \eqref{mM.grad.commutator}. 
\end{proof}

Define 
\begin{equation}  \label{def.mI}
\mI(\eta, \beta) := \int_{\S^2} (1+\eta)^2 (\mM \eta) \beta \, d\sigma
= - \frac13 \int_{\S^2} (1+\eta)^3 \mM \beta \, d\sigma, 
\end{equation}
where the last identity in \eqref{def.mI} follows from formula \eqref{mM.is.anti.symm}. 
By \eqref{def.mI}, 
the gradients of $\mI$ with respect to the $L^2(\S^2)$ scalar product are 
\begin{equation} 
\label{pa.mI}
\pa_\eta \mI(\eta, \beta) = - (1 + \eta)^2 \mM \beta,  
\quad \  
\pa_\beta \mI(\eta, \beta) = (1+\eta)^2 \mM \eta.
\end{equation}

Using \eqref{pa.mI} to substitute $\mM \eta, \mM \beta$, 
system \eqref{quasi.Ham.syst.eta.beta} becomes 
\[
\pa_\beta \mH_{\sigma_0}(\eta, \beta) - \om \pa_\beta \mI(\eta, \beta) = 0, 
\quad \ 
\pa_\eta \mH_{\sigma_0}(\eta, \beta) - \om \pa_\eta \mI(\eta, \beta) = 0, 
\]
that is, 
\begin{equation} \label{eq.critical.point}
\grad \mH_{\sigma_0}(\eta, \beta) 
- \om \grad \mI(\eta, \beta) = 0,
\end{equation}
where $\grad = (\pa_\eta, \pa_\beta)$ is the gradient 
with respect to the $L^2(\S^2, \R^2)$ scalar product.
Thus \eqref{eq.critical.point} is the equation of the critical points 
of the functional $\mH_{\sigma_0}(\eta, \beta) - \om \mI(\eta, \beta)$. 

Before starting with the study of equation \eqref{eq.critical.point}, 
in the next section we study some symmetries of the Hamiltonian $\mH$ in \eqref{def.mH}
and conserved quantities of system \eqref{quasi.Ham.syst.2.sigma.0.bis}.

\section{Symmetries, prime integrals and group action} 
\label{sec:symmetries}

In this section we study some symmetries of the Hamiltonians $\mH$ in \eqref{def.mH} 
and $\mH_{\sigma_0} = \mH - 2 \sigma_0 \mV$ in \eqref{def.mH.sigma.0} 
and the prime integrals corresponding to them by Noether's theorem, which are 

- the angular momentum $\mI(h,\psi)$ with respect to the $x_3$ axis (Lemma \ref{lem:conservation}); 

- the mass, or volume, $\mV(h)$ (Lemma \ref{lemma:mass});

- the barycenter velocity $\mB(h,\psi)$ (Lemma \ref{lem.invariance}). 

Then we observe an invariance property under a torus action (Lemma \ref{lemma:group.action}).

\medskip

It is convenient to introduce the Poisson brackets notation 
corresponding to the Hamiltonian structure of system \eqref{quasi.Ham.syst.2.sigma.0.bis}: 
given any two functionals $A(h,\psi)$, $B(h,\psi)$, we denote 
\begin{equation} \label{Poisson.brackets} 
\{ A, B \}(h, \psi) := \la \pa_h A(h,\psi) , \frac{\pa_\psi B(h,\psi)}{(1+h)^2} \ra_{L^2(\S^2)} 
- \la \pa_\psi A(h,\psi) , \frac{\pa_h B(h,\psi)}{(1+h)^2} \ra_{L^2(\S^2)},
\end{equation}
where $\pa_h A, \pa_\psi A, \pa_h B, \pa_\psi B$ are the gradients of $A, B$ 
with respect to the scalar product of $L^2(\S^2)$. 
We recall that, if $\{ A, \mH_{\s_0} \} = 0$, then, 
by the chain rule, the time derivative of $A(h(t),\psi(t))$ vanishes 
along any solution $(h(t), \psi(t))$ of \eqref{quasi.Ham.syst.2.sigma.0.bis},  
and $A$ is a prime integral of system \eqref{quasi.Ham.syst.2.sigma.0.bis}.

Now we show first that $\mI$ in \eqref{def.mI} is a prime integral of \eqref{quasi.Ham.syst.2.sigma.0.bis}
associated to the invariance of $\mH_{\sigma_0}$ with respect to the action of the rotations $R(\th)$.

\begin{lemma} \label{lem:conservation}
Let $R(\th)$, $\th \in \T := \R / 2 \pi \Z$, be the rotation matrix \eqref{def.R(th)}, 
and let $\mT_{\th}$ denote the composition operator 
\begin{equation}  \label{def.mT.theta}
\mT_\th f = f \circ R(\th), 
\end{equation}
that is, $(\mT_\th f)(x) = f(R(\th)x)$ for all $x \in \S^2$,
for any function $f$ defined on $\S^2$. Then

$(i)$ the Hamiltonian $\mathcal{H}$ in \eqref{def.mH} 
and the volume $\mV$ in \eqref{def.volume}
are $\mT_\th$-invariant, i.e., 
\begin{equation}  \label{Ham.is.invariant}
\mH \circ \mT_\th = \mH,
\quad 
\mV \circ \mT_\th = \mV,
\end{equation}
that is, $\mH( \mT_\th h, \mT_\th \psi) = \mH( h, \psi )$ 
and $\mV(\mT_\th h) = \mV(h)$ 
for all $(h, \psi)$; 

$(ii)$ the infinitesimal generator $\tfrac{d}{d\th} \big|_{\th=0} R(\th)$ 
of the one-parameter group of rotations $\{ R(\th) : \th \in \T \}$  
is the matrix $\mJ$ in \eqref{def.mJ};

$(iii)$ the infinitesimal generator $\tfrac{d}{d\th} \big|_{\th=0} \mT_\th$ 
of the one-parameter group of composition operators $\{ \mT_\th : \th \in \T \}$
is the differential operator $\mM$ in \eqref{def.mM};

$(iv)$ the functional $\mI$ defined in \eqref{def.mI} 
satisfies $\{ \mI , \mH \} = 0$ and $\{ \mI, \mV \} = 0$, 
it is a prime integral of the Hamiltonian system \eqref{quasi.Ham.syst.2.sigma.0.bis},  
and it is the angular momentum associated to the rotation invariance \eqref{Ham.is.invariant}. 
\end{lemma}

\begin{proof} 
$(i)$ As noticed in Lemma 6.2 of \cite{B.J.LM}, one has 
\[
G( \mT_\th h)[ \mT_\th \psi] = \mT_\th ( G(h)\psi ), 
\quad \  
| \grad_{\S^2} (\mT_\th h) | = | \mT_\th (\grad_{\S^2} h) |. 
\]
Hence \eqref{Ham.is.invariant} follows from the change of integration variable 
$R(\th) x = y$ in formulas \eqref{def.mH} and \eqref{def.volume}. 
$(ii)$ By \eqref{def.R(th)}, the identity is elementary. 
$(iii)$ The identity follows from $(ii)$ and the chain rule. 
$(iv)$ By \eqref{Ham.is.invariant}, one has $\tfrac{d}{d\th} \mH(\mT_\th h, \mT_\th \psi) = 0$. 
This identity at $\th = 0$, using the chain rule and $(iii)$, gives 
\begin{equation} \label{Noether.temp.01}
\la \pa_h \mH(h,\psi) , \mM h \ra_{L^2(\S^2)} 
+ \la \pa_\psi \mH(h,\psi) , \mM \psi \ra_{L^2(\S^2)}
= 0.
\end{equation}
Using \eqref{pa.mI}, we deduce that $\{ \mI , \mH \} = 0$. 
Similarly, one proves that $\{ \mI , \mV \} = 0$.
\end{proof}

Second, in the next lemma we prove that $\mV$ itself 
is a prime integral of system \eqref{quasi.Ham.syst.2.sigma.0.bis}.
This is closely related to the invariance already observed in the lines 
surrounding equation \eqref{syst.pat.h.pat.psi.2.sigma.0}, 
corresponding to the basic fact that the velocity potential $\Phi$ is defined up to a constant.

\begin{lemma} \label{lemma:mass}
Let $\mS_a$, $a \in \R$, denote the shift operator 
\[
\mS_a (h,\psi) = (h, \psi + a).
\]
Then
$(i)$ the Hamiltonian $\mathcal{H}$ in \eqref{def.mH} is $\mS_a$-invariant, i.e., 
\begin{equation}  \label{Ham.is.invariant.shift}
\mH \circ \mS_a = \mH;
\end{equation}

$(ii)$ the infinitesimal generator $\tfrac{d}{da} \big|_{a=0} \mS_a$ 
of the one-parameter group of operators $\{ \mS_a : a \in \R \}$
is the constant map $(h,\psi) \mapsto (0,1)$;

$(iii)$ the functional $\mV$ defined in \eqref{def.volume} 
satisfies $\{ \mV , \mH \} = 0$, 
it is a prime integral of the Hamiltonian system \eqref{quasi.Ham.syst.2.sigma.0.bis},  
and it is the fluid mass, associated to the invariance \eqref{Ham.is.invariant.shift}.
\end{lemma}

\begin{proof} 
$(i)$ By definition \eqref{def:Dirichlet-Neumann}, one has $G(h) 1 = 0$ and, by linearity, 
$G(h)(\psi + a) = G(h) \psi$. Moreover, 
\[
\int_{\S^2} \psi_1 (G(h)\psi_2) (1 + h) J \, d\sigma 
= \int_{\S^2} \psi_2 (G(h)\psi_1) (1 + h) J \, d\sigma 
\]
for all $\psi_1, \psi_2$, with $J$ in \eqref{def.J}, see (3.22) in \cite{B.J.LM}. 
Hence \eqref{Ham.is.invariant.shift} follows from formula \eqref{def.mH}. 
$(ii)$ Trivial. 
$(iii)$ By \eqref{Ham.is.invariant.shift}, one has $\tfrac{d}{da} \mH(\mS_a (h,\psi)) = 0$. 
This identity at $a = 0$, by the chain rule and $(ii)$, gives 
\begin{equation} \label{Noether.temp.02}
\la \pa_h \mH(h,\psi) , 0 \ra_{L^2(\S^2)} 
+ \la \pa_\psi \mH(h,\psi) , 1 \ra_{L^2(\S^2)}
= 0.
\end{equation}
Hence, by \eqref{pa.mV}, $\{ \mV , \mH \} = 0$. 
Also $\{ \mV , \mV \}$ trivially vanishes.
\end{proof}

Third, we consider the quantity
\begin{equation}  \label{def.mB}
\mB(h,\psi) := \frac12 \int_{\S^2} (1+h)^2 \grad_{\S^2} \psi \, d\sigma
= \int_{\S^2} x (1+h)^2 \psi \, d\sigma 
- \int_{\S^2} (1+h) \psi \grad_{\S^2} h \, d\sigma,
\end{equation}
which is the fluid barycenter velocity multiplied by the volume of $\Om_t$, 
written in terms of $(h, \psi)$. 
Note that $\mB(h,\psi)$ is a vector in $\R^3$. 
The second identity in \eqref{def.mB} is obtained by applying the divergence theorem on $\S^2$ 
\eqref{div.thm.S2} to the vector field $F = (1 + h)^2 \psi e_k$ for $k = 1,2,3$. 
By \eqref{def.mB}, the gradients of $\mB$ with respect to the $L^2(\S^2)$ scalar product are 
\begin{equation} \label{pa.mB}
\pa_h \mB(h,\psi) = (1+h) \grad_{\S^2} \psi, 
\quad \ 
\pa_\psi \mB(h,\psi) = x (1+h)^2 - (1+h) \grad_{\S^2} h.
\end{equation}
The symmetry of $\mH$ associated to $\mB$ is the translation invariance 
of the original liquid drop problem, 
whose definition, in our geometric formulation where the fluid domain is parametrized 
as a graph over the (fixed, i.e., non translated) unit sphere, is nontrivial. 
In the next lemma we prove that, for any small elevation function $h$, 
the translated boundary $\pa \Om + \alpha = \{ x (1 + h(x)) + \alpha : x  \in \S^2 \}$ 
is also parametrized as $\{ \xi ( 1 + h_\alpha(\xi)) : \xi \in \S^2 \}$, 
where $h_\alpha$ is another suitable elevation function.

\begin{lemma} 
\label{lemma:costruz.h.alpha}
Let $\a \in \R^3$, $|\a| < \tfrac14$, 
and let $h \in C^1(\S^2, \R)$, $\| h \|_{W^{1,\infty}(\S^2)} < \frac16$. 

\medskip

$(i)$ For every $x \in \S^2$ there exists a unique pair $(\xi, \lm) \in \S^2 \times (0, \infty)$ 
such that 
\begin{equation} \label{eq.x.xi.lm}
(1 + h(x)) x + \alpha = \lm \xi,
\end{equation}
which is 
\begin{equation} \label{formula.xi.lm}
\xi = \frac{ (1 + h(x)) x + \alpha }{ |(1 + h(x)) x + \alpha| }, \quad \ 
\lm = |(1 + h(x)) x + \alpha|.
\end{equation}

\medskip

$(ii)$ For every $\xi \in \S^2$ there exists a unique pair $(x, \lm) \in \S^2 \times (0, \infty)$ 
satisfying \eqref{eq.x.xi.lm}.

\medskip

$(iii)$ Let $f : \S^2 \to \S^2$, $x \mapsto f(x)$ 
be the function mapping any point $x \in \S^2$ to the point $\xi$ given by item $(i)$,  
and let $g : \S^2 \to \S^2$, $\xi \mapsto g(\xi)$  
be the function mapping any point $\xi \in \S^2$ to the point $x$ given by item $(ii)$. 
Both $f$ and $g$ are bijective maps of $\S^2$, 
and one is the inverse of the other.

\medskip

$(iv)$ Let $h_\a : \S^2 \to \R$, $\xi \mapsto h_\a (\xi)$ 
be the function mapping any point $\xi \in \S^2$ into $\lm - 1$, 
where $\lm$ is the real number given by item $(ii)$. 
Then 
\begin{align} \label{eq.h.alpha.xi}
(1 + h(g(\xi))) g(\xi) + \alpha = (1 + h_\a(\xi)) \xi, \quad \ 
(1 + h(x)) x + \alpha = (1 + h_\a( f(x) )) f(x) 
\end{align}
for all $x, \xi \in \S^2$. 
The functions $f, g, h_\alpha$ are determined by $\alpha, h$, 
in the sense that there are no other functions satisfying \eqref{eq.h.alpha.xi}.

\medskip

$(v)$ One has $f, g \in C^1(\S^2,\S^2)$, $h_\a \in C^1(\S^2,\R)$, with 
\begin{align} 
\| f - \mathrm{id} \|_{W^{1,\infty}(\S^2)}
+ \| g - \mathrm{id} \|_{W^{1,\infty}(\S^2)} 
+ \| h_\alpha - h \circ g \|_{W^{1,\infty}(\S^2)}
& \leq C |\a|. 
\label{est.h.alpha}
\end{align}
 \end{lemma}

\begin{proof} 
$(i)$ Let $x \in \S^2$. One has 
\[
|(1 + h(x)) x + \alpha| 
\geq |(1 + h(x))x| - |\alpha| 
\geq 1 - \| h \|_\infty - |\alpha| 
> \tfrac12.
\]
Hence the vector $(1 + h(x)) x + \alpha$ is nonzero, 
and therefore the pair $(\xi, \lm)$ in \eqref{formula.xi.lm} 
is well defined, it belongs to $\S^2 \times (0,\infty)$ 
and it solves \eqref{eq.x.xi.lm}. 
If $(\xi', \lm') \in \S^2 \times (0, \infty)$ is another solution of \eqref{eq.x.xi.lm}, 
then $\lm \xi = \lm' \xi'$, 
whence $\lm = |\lm \xi| = |\lm' \xi'| = \lm'$, 
and therefore also $\xi = \xi'$.

$(ii)$ First, we prove a necessary condition for $\lm$. 
Suppose that $x, \xi \in \S^2$, $\lm \in (0, \infty)$ satisfy \eqref{eq.x.xi.lm}.
Then 
\begin{equation}  \label{S2.cond.phi}
| \lm \xi - \alpha | = | (1 + h(x)) x | = 1 + h(x),
\end{equation}
and, taking the square, 
\[
\lm^2 - 2 \lm \la \xi , \alpha \ra + |\a|^2 - (1 + h(x))^2 = 0.
\]
This is a polynomial in $\lm$ of degree 2 with discriminant
\begin{equation} \label{bound.Delta}
\tfrac14 \Delta = \la \xi , \alpha \ra^2 - |\a|^2 + (1 + h(x))^2 
\geq \la \xi , \alpha \ra^2 - |\a|^2 + (1 - \| h \|_\infty)^2 
\geq \la \xi , \alpha \ra^2 + \tfrac12,
\end{equation}
having two distinct real roots of opposite sign. 
Hence $\lm$ is the unique real positive root
\begin{equation} \label{formula.lm}
\lm = \la \xi , \alpha \ra + \sqrt{ \la \xi, \alpha \ra^2 - |\alpha|^2 + (1 + h(x))^2 }.
\end{equation}
We also observe that, vice versa, \eqref{formula.lm} implies \eqref{S2.cond.phi}, 
that is, given $x, \xi \in \S^2$, 
if $\lm$ is the number in \eqref{formula.lm},  
then identity \eqref{S2.cond.phi} holds.

Now we use the necessary condition \eqref{formula.lm} to solve equation \eqref{eq.x.xi.lm}
with $\xi$ given and $(x,\lm)$ as unknown.  
Let $\xi \in \S^2$. For any $x \in \S^2$, denote $r(x)$ the right-hand side of \eqref{formula.lm}. 
A pair $(x, \lm) = (x, r(x))$ solves equation \eqref{eq.x.xi.lm} if and only if 
$x$ solve the fixed point equation 
\begin{equation} \label{fix.pt.eq}
x = \phi(x), \quad \ 
\phi(x) := \frac{ r(x) \xi - \alpha }{ 1+h(x) }.
\end{equation}
For all $x \in \S^2$ one has $|\phi(x)| = 1$ because \eqref{formula.lm} implies \eqref{S2.cond.phi}.
Hence $\phi$ is a map of $\S^2$ into itself. 
We consider $\S^2$ as the complete metric space where the distance between two points $x, y \in \S^2$ 
is the $\R^3$ Euclidean norm $|x-y|$, and we show that $\phi$ is a contraction mapping on $\S^2$.  
For all $x, y \in \S^2$, one has 
\begin{align*}
\phi(x) - \phi(y) 
& = \frac{ ( r(x) - r(y) ) \xi }{ 1+h(x) } 
+ \frac{ ( r(y) \xi - \alpha ) (h(y) - h(x)) }{ (1 + h(x)) (1 + h(y)) },
\\
|\phi(x) - \phi(y)| 
& \leq \frac{ |r(x) - r(y)| + |h(y) - h(x)| }{ 1 + h(x) }
\end{align*}
because $|r(y) \xi - \alpha| = 1 + h(y)$, 
and
\begin{align*}
r(x) - r(y) & = \frac{ (2 + h(x) + h(y)) (h(x) - h(y)) }
{ \sqrt{ \la \xi, \alpha \ra^2 - |\alpha|^2 + (1 + h(x))^2 } 
+ \sqrt{ \la \xi, \alpha \ra^2 - |\alpha|^2 + (1 + h(y))^2 }},
\\ 
|r(x) - r(y)| & \leq 3 | h(x) - h(y) |
\end{align*}
by \eqref{bound.Delta}. Hence 
\begin{equation} \label{lip.phi.lip.h}
|\phi(x) - \phi(y)| \leq 6 |h(y) - h(x)|.
\end{equation} 
Moreover, 
\begin{equation} \label{h.lip}
|h(y) - h(x)| \leq \| h \|_{W^{1,\infty}(\S^2)} |x-y|.
\end{equation}
To prove \eqref{h.lip}, consider the 1-homogeneous extension $\mE_1 h(x) := |x| \mE_0 h(x)$, 
integrate the derivative of the function $s \mapsto \mE_1 h (y + s(x-y))$ over the interval $[0,1]$, 
use the fact that $\grad \mE_1 h$ is $0$-homogeneous to obtain 
$\| \grad \mE_1 h \|_{L^\infty(B_1)} = \| \grad \mE_1 h \|_{L^\infty(\S^2)}$, 
and $\grad \mE_1 h(x) = \grad_{\S^2} h(x) + h(x) x$ on $\S^2$, 
$|\grad \mE_1 h| \leq |\grad_{\S^2} h| + |h|$ on $\S^2$. 

By \eqref{lip.phi.lip.h} and \eqref{h.lip}, 
for $\| h \|_{W^{1,\infty}(\S^2)} < \tfrac16$, 
$\phi$ is a contraction map on the complete metric space $(\S^2, | \cdot |)$, 
and therefore it has a unique fixed point $x = \phi(x) \in \S^2$. 
Then, by the definition of $\phi$ in \eqref{fix.pt.eq}, 
the pair $(x,\lm) = (x, r(x))$ is a solution of \eqref{eq.x.xi.lm}.

To prove the uniqueness, assume that a pair $(x,\lm) \in \S^2 \times (0,\infty)$ solves \eqref{eq.x.xi.lm}. 
Then, as proved above, $\lm$ must be the number in \eqref{formula.lm}, 
that is, $\lm = r(x)$, and therefore, by identity \eqref{eq.x.xi.lm},
$x$ is a fixed point of $\phi$. Hence $(x,\lm)$ is the pair $(x, r(x))$ where $x = \phi(x)$.

$(iii)$ 
Let $\S^2 \to \S^2 \times (0, \infty)$, $x \mapsto (f(x), \lm_0(x))$ 
be the function that maps any $x \in \S^2$ to the unique solution $(\xi, \lm)$ 
of \eqref{eq.x.xi.lm}, whose existence and uniqueness is given by item $(i)$.  
Similarly, let $\S^2 \to \S^2 \times (0, \infty)$, $\xi \mapsto (g(\xi), \lm_1(\xi))$ 
be the function that maps any $\xi \in \S^2$ to the unique solution $(x, \lm)$ 
of \eqref{eq.x.xi.lm}, whose existence and uniqueness is given by item $(ii)$.  
For every $x \in \S^2$, the pair $(\xi, \lm) = (f(x), \lm_0(x))$ solves \eqref{eq.x.xi.lm}. 
Hence the pair $(x,\lm) = (x, \lm_0(x))$ is the unique solution $(x,\lm)$ of equation \eqref{eq.x.xi.lm} 
with datum $\xi = f(x)$, that is, $(x,\lm_0(x)) = (g(f(x)), \lm_1(f(x)))$. 
Therefore $g(f(x)) = x$ and $\lm_1(f(x)) = \lm_0(x)$ for all $x \in \S^2$.
Similarly one proves that $f(g(\xi)) = \xi$ 
and $\lm_1(\xi) = \lm_0(g(\xi))$ for all $\xi \in \S^2$. 
Hence $f$ and $g$ are invertible maps of $\S^2$, and one is the inverse of the other.

$(iv)$ By construction, the functions $f,\lm_0,g,\lm_1$ defined above satisfy
\[
(1 + h(x)) x + \alpha = \lm_0(x) f(x), 
\quad \ 
( 1 + h(g(\xi)) ) g(\xi) + \alpha = \lm_1(\xi) \xi
\]
for all $x, \xi \in \S^2$. 
The function $h_\alpha$ in the statement is $h_\alpha(\xi) = \lm_1(\xi) - 1$, 
and therefore \eqref{eq.h.alpha.xi} is satisfied. 
The uniqueness of $f, g, h_\alpha$ directly follows from the construction.

$(v)$ It follows from identities \eqref{eq.h.alpha.xi} and the construction above.
\end{proof}

For $\a \in \R^3$, $|\a| < \tfrac14$, we denote $\mA_\a$ the operator 
\begin{equation} \label{def.mA.alpha}
\mA_\a (h,\psi) := (h_\alpha, \psi \circ g_\alpha),
\end{equation}
where $h_\alpha$ is the function in item $(iv)$ of Lemma \ref{lemma:costruz.h.alpha}, 
and $g_\alpha$ is the function $g$ in item $(iii)$ of the same lemma, 
where the dependence on $\alpha$ was not explicitly indicated.

\begin{lemma}
Let $\a, \b \in \R^3$, with $|\a|, |\b|, |\a+\b| < \frac14$. 
Then 
\[
g_\alpha \circ g_\beta = g_{\alpha + \beta}, \quad \ 
(h_\alpha)_\beta = h_{\alpha + \beta}, \quad \  
\mA_\a \circ \mA_\b = \mA_{\a + \b}.
\] 
\end{lemma}

\begin{proof}
By \eqref{eq.h.alpha.xi} with $\beta$ instead of $\alpha$, one has 
\begin{equation} \label{group.01}
(1 + h \circ g_\beta(\xi)) g_\beta(\xi) + \beta 
= (1 + h_\beta(\xi)) \xi
\end{equation}
for all $\xi \in \S^2$. 
By \eqref{eq.h.alpha.xi} applied with $h_\beta$ in the role of $h$, one has
\begin{equation} \label{group.02}
(1 + h_\beta \circ g_\alpha(\th)) g_\alpha(\th) + \alpha
= (1 + (h_\beta)_\alpha(\th)) \th 
\end{equation}
for all $\th \in \S^2$. 
Identity \eqref{group.01} at $\xi = g_\alpha(\th)$ gives
\begin{equation} \label{group.03}
\big( 1 + h \circ g_\beta \circ g_\alpha(\th) \big) g_\beta \circ g_\alpha(\th) + \beta 
= \big(1 + h_\beta \circ g_\alpha(\th) \big) g_\alpha(\th)
\end{equation}
for all $\th \in \S^2$. 
The left-hand side of \eqref{group.02} 
is the sum of the right-hand side of \eqref{group.03} with $\alpha$, 
whence 
\[
\big( 1 + h \circ g_\beta \circ g_\alpha(\th) \big) g_\beta \circ g_\alpha(\th) + \beta + \alpha 
= (1 + (h_\beta)_\alpha(\th)) \th 
\]
for all $\th \in \S^2$. 
On the other hand, by \eqref{eq.h.alpha.xi} with $\alpha + \beta$ instead of $\alpha$, 
one also has 
\[
\big( 1 + h \circ g_{\alpha + \beta} (\th) \big) g_{\alpha + \beta} (\th) + \alpha + \beta 
= (1 + h_{\alpha + \beta}(\th) ) \th 
\]
for all $\th \in \S^2$. 
Hence, by the uniqueness property in item $(iv)$ of Lemma \ref{lemma:costruz.h.alpha}, 
$g_\beta \circ g_\alpha = g_{\alpha + \beta}$ 
and $h_{\alpha + \beta} = (h_\beta)_\alpha$. 
Since the sum is commutative, the lemma is proved. 
\end{proof}

\begin{lemma}\label{lem.translations}
One has 
\begin{equation}  \label{pa.mA}
\frac{d}{d\alpha} \Big|_{\alpha = 0} \mA_\a (h,\psi)
= \Big( \frac{ (1+h)x - \grad_{\S^2} h }{ 1+h } \,, \, 
- \frac{ \grad_{\S^2} \psi }{ 1+h } \Big).
\end{equation}
\end{lemma}

\begin{proof}
From \eqref{formula.xi.lm} and item $(iii)$ in Lemma \ref{lemma:costruz.h.alpha},
we have the explicit formula 
\begin{equation} \label{formula.f.alpha}
f_\alpha(x) = \frac{ (1 + h(x)) x + \alpha }{ |(1 + h(x)) x + \alpha| }.
\end{equation}
Hence we calculate the derivatives of $f_\alpha$ at $\alpha = 0$
with respect to $\alpha$ in direction $\beta \in \R^3$ 
and with respect to $x$ in direction $v \in T_x \S^2$,  
\[
\frac{d}{d\alpha} \Big|_{\alpha = 0} f_\alpha(x)[\beta] 
= \frac{I - x \otimes x}{1+h(x)} \beta
= \frac{\beta - x \la x , \beta \ra}{1+h(x)}, 
\quad \ 
D f_\alpha(x) \big|_{\alpha = 0}[v] = (I - x \otimes x)v = v.
\]
The last identity is immediate because, at $\alpha = 0$, 
$f_\alpha$ is the identity map of $\S^2$.
Concerning $g_\alpha$, one has the identity $f_\a(g_\a(\xi)) = \xi$ on $\S^2$, 
which, with a temporary, helpful notation, can also be written as 
$f( \alpha , g (\alpha , \xi)) = \xi$.    
Differentiating this identity, and using the fact that, for $\a = 0$, $g_\alpha$ 
is the identity map of $\S^2$, we get  
\[
\frac{d}{d\alpha} \Big|_{\alpha = 0} g_\alpha(\xi) [\beta]
= - \frac{I - \xi \otimes \xi}{1+h(\xi)} \beta 
= \frac{ - \beta + \xi \la \xi , \beta \ra}{1+h(\xi)},
\quad \ 
D g_\alpha(\xi) \big|_{\alpha = 0}[v] = (I - \xi \otimes \xi) v = v.
\]
Also, we have the chain rule
\begin{equation} \label{chain.rule.h}
\frac{d}{d\alpha} \Big|_{\alpha = 0} h \circ g_\alpha(\xi) [\beta]
= \la \grad_{\S^2} h(\xi) , \frac{d}{d\alpha} \Big|_{\alpha = 0} g_\alpha(\xi) [\beta] \ra 
= - \frac{ \la \grad_{\S^2} h(\xi) , \beta \ra }{ 1+h(\xi) }.
\end{equation}
Recalling \eqref{eq.h.alpha.xi}, one has 
\begin{equation} \label{recall.eq.h.alpha.xi}
(1 + h \circ g_\alpha(\xi)) g_\alpha(\xi) + \alpha = (1 + h_\a(\xi)) \xi
\end{equation}
on $\S^2$. Differentiating this identity with respect to $\a$ in direction $\beta$ at $\alpha = 0$
we get 
\[
- \frac{ \la \grad_{\S^2} h(\xi) , \beta \ra }{ 1+h(\xi) } \xi 
+ (1 + h(\xi)) \frac{ - \beta + \xi \la \xi , \beta \ra}{1+h(\xi)} 
+ \beta = \frac{d}{d \alpha} \Big|_{\alpha = 0}  h_\a (\xi)[\beta] \, \xi,
\]
namely
\[
\frac{d}{d \alpha} \Big|_{\alpha = 0}  h_\a (\xi)[\beta]
= \la \frac{ (1+h(\xi)) \xi - \grad_{\S^2} h(\xi) }{ 1+h(\xi) } , \beta \ra. 
\]
This gives the first component in the statement. 
The second component follows from \eqref{chain.rule.h} with $\psi$ instead of $h$. 
\end{proof}

We note that at $\alpha = 0$ one has $h_\alpha = h$, 
and therefore the derivative of $h_\alpha(\xi)$ with respect to $\xi$ in direction $v \in T_\xi \S^2$ 
is 
\[
D h_\alpha (\xi) \big|_{\alpha = 0} [v] 
= D h (\xi)[v] 
= \la \grad_{\S^2} h(\xi) , v \ra.
\]

\begin{lemma} \label{lem.invariance}
Let $\mA_\alpha$ be the operator defined in \eqref{def.mA.alpha} 
for all $\a \in \R^3$, $|\a| < \frac14$. Then 

$(i)$ the Hamiltonian $\mathcal{H}$ in \eqref{def.mH} and the volume $\mV$ in \eqref{def.volume}
are  $\mA_\a$-invariant, i.e., 
\begin{equation}  \label{Ham.is.invariant.mA}
\mH \circ \mA_\a = \mH, \quad 
\mV \circ \mA_\a = \mV;
\end{equation}

$(ii)$ the functional $\mB$ defined in \eqref{def.mB} 
satisfies $\{ \mB , \mH \} = 0$ and $\{ \mB , \mV \} = 0$, 
it is a prime integral of the Hamiltonian system \eqref{quasi.Ham.syst.2.sigma.0.bis},  
and it is the fluid barycenter velocity, 
associated to the invariance \eqref{Ham.is.invariant.mA}.
\end{lemma}

\begin{proof} 
$(i)$ By Section 3.2 of \cite{B.J.LM}, one has 
\begin{equation} \label{physical.Ham}
\mH(h,\psi) = \frac12 \int_\Om |\grad \Phi|^2 \, dx + \sigma_0 \mathrm{Area}(\pa \Om),
\quad \ \mV(h) = \int_\Om dx = \mathrm{Volume}(\Om),
\end{equation}
where $\Om$ is the bounded domain with boundary $\pa \Om = \{ (1 + h(x)) x : x \in \S^2 \}$
and $\Phi$ is the solution in $H^1(\Om)$ of problem \eqref{def:Dirichlet-Neumann2}. 
Let $\Om_\alpha := \Om + \alpha$. 
Its boundary is $\pa \Om_\alpha = \pa \Om + \alpha$, 
and, by the translation invariance of the Hausdorff measure and that of the Lebesgue measure, 
\[
\mathrm{Area}( \pa \Omega_\alpha ) = \mathrm{Area}( \pa \Omega ), 
\quad \ 
\mathrm{Volume}( \Omega_\alpha ) = \mathrm{Volume}( \Omega ).
\]
Let $\Phi_\alpha (x) := \Phi(x-\alpha)$. Then $\Phi_\alpha \in H^1(\Om_\alpha)$, 
$\grad \Phi_\alpha(x) = (\grad \Phi)(x - \alpha)$ for all $x \in \Omega_\alpha$, 
and, with the change of variable $x - \alpha = y$, one has
\[
\int_{\Omega_\alpha} |\grad \Phi_\alpha|^2 \, dx = \int_{\Omega} |\grad \Phi|^2 \, dx. 
\]
Moreover $\Delta \Phi_\alpha = 0$ in $\Omega_\alpha$ and 
\[
\Phi_\alpha( (1 + h(x)) x + \alpha ) 
= \Phi (1 + h(x)) x ) 
= \psi(x)  
\quad \forall x \in \S^2.
\]
The diffeomorphism $g_\alpha$ of $\S^2$ satisfies \eqref{recall.eq.h.alpha.xi}, 
and therefore the change of variable $x = g_\alpha(\xi)$ 
gives the identities 
\[
\pa \Om_\alpha = \{ (1 + h_\alpha(\xi)) \xi : \xi \in \S^2 \}, 
\quad \  
\Phi_\alpha ( (1 + h_\alpha(\xi)) \xi ) = \psi \circ g_\alpha (\xi) 
\quad \ \forall \xi \in \S^2.
\]
Hence formulas \eqref{physical.Ham} applied with $\Om_\alpha, h_\alpha, \psi \circ g_\alpha$ 
in place of $\Omega, h,\psi$ give
\[
\mH (h_\alpha, \psi \circ g_\alpha) = 
\frac12 \int_{\Om_\alpha} |\grad \Phi_\alpha|^2 \, dx + \sigma_0 \mathrm{Area}(\pa \Om_\alpha), 
\quad \ 
\mV(h_\alpha) = \mathrm{Volume}(\Om_\alpha).
\]
This proves that $\mH ( \mA_\alpha (h, \psi)) = \mH(h,\psi)$ and 
$\mV ( \mA_\alpha (h, \psi)) = \mV(h,\psi)$.

$(ii)$ By \eqref{Ham.is.invariant.mA}, one has $\tfrac{d}{d\alpha} \mH(\mA_\alpha (h, \psi)) = 0$. 
This identity at $\alpha = 0$, using the chain rule and \eqref{pa.mA}, gives 
\begin{equation} \label{Noether.temp.05}
\la \pa_h \mH(h,\psi) , \frac{(1+h) x - \grad_{\S^2} h}{1+h} \ra_{L^2(\S^2)} 
+ \la \pa_\psi \mH(h,\psi) , - \frac{\grad_{\S^2} \psi}{1+h} \ra_{L^2(\S^2)}
= 0.
\end{equation}
By \eqref{pa.mB}, we deduce that $\{ \mB , \mH \} = 0$. 
Similarly, one proves that $\{ \mB , \mV \} = 0$. 
\end{proof}

As a consequence of Lemma \ref{lem:commutation.formula} one also has 
$\{ \mI, \mV \}=0$ and $\{\mI, \mB_3\}=0$, 
as we prove in the next Lemma. 

\begin{lemma}\label{lem:poisson.Mi.MB}
One has 
\begin{equation} \label{eq:poisson.mI.mB}
\{\mI, \mB\} = \mJ \mB, \quad \ 
\{ \mI, \mB_1 \} = - \mB_2, \quad \ 
\{ \mI, \mB_2 \} = \mB_1, \quad \ 
\{ \mI, \mB_3 \} = 0,
\end{equation}
where $\mI, \mB, \mJ$ are defined in \eqref{def.mI}, 
\eqref{def.mB}, \eqref{def.mJ}, and $\mB_k$ is the $k$-th component of $\mB$.
\end{lemma}

\begin{proof}
Let $f := \frac12 (1+h)^2$. Then 
$\grad_{\S^2} f = (1+h) \grad_{\S^2} h$ 
and $\mM f = (1+h) \mM h$. 
By definition \eqref{Poisson.brackets} 
and formulas \eqref{pa.mB} and \eqref{pa.mI}, 
one has 
\begin{align} 
\{ \mB, \mI \}(h, \psi)
& = \int_{\S^2} (\mM \psi) \big[ x (1+h)^2 - (1+h) \grad_{\S^2} h \big] \, d\sigma
+ \int_{\S^2} (\mM h) (1+h) \nabla_{\S^2} \psi \, d\sigma
\notag \\
& = \int_{\S^2} 2x f \mM \psi \, d\sigma
- \int_{\S^2} (\mM \psi) \grad_{\S^2} f \, d\sigma
+ \int_{\S^2} (\mM f) \nabla_{\S^2} \psi \, d\sigma.
\label{eq:poisson.mB.mI}
\end{align}
Given any scalar function $g$, the divergence theorem \eqref{div.thm.S2} 
applied to the vector field $F = e_k g$, $k=1,2,3$, gives the identity
\begin{equation} \label{int.x.is.int.grad}
\int_{\S^2} 2x g \, d\sigma = \int_{\S^2} \grad_{\S^2} g \, d\sigma.
\end{equation}
Applying \eqref{int.x.is.int.grad} to the function $g = f \mM \psi$, 
the first integral in the right hand side of \eqref{eq:poisson.mB.mI} is 
\begin{equation}\label{eq:poisson.mB.mI1}
\int_{\S^2} 2x f \mM \psi \, d\sigma 
= \int_{\S^2} \grad_{\S^2} (f \mM \psi) \, d\sigma 
= \int_{\S^2} (\grad_{\S^2} f) (\mM \psi)  \, d\sigma
+ \int_{\S^2} f \grad_{\S^2} (\mM \psi) \, d\sigma.
\end{equation} 
Combining \eqref{eq:poisson.mB.mI} and \eqref{eq:poisson.mB.mI1} we find that
\begin{align*} 
\{ \mB , \mI \}(h,\psi) 
& = \int_{\S^2} f \grad_{\S^2} (\mM \psi) \, d\sigma
+ \int_{\S^2} (\mM f) \nabla_{\S^2} \psi \, d\sigma
\\
& = \int_{\S^2} f \grad_{\S^2} (\mM \psi) \, d\sigma
- \int_{\S^2} f \mM( \nabla_{\S^2} \psi) \, d\sigma,
\end{align*}
where in the last identity we have used \eqref{mM.is.anti.symm}. 
Then, from the commutator identity \eqref{mM.grad.commutator}, 
we obtain $\{ \mB, \mI \} = - \mJ \mB$, which gives the first identity in \eqref{eq:poisson.mI.mB}.
The other three identities in \eqref{eq:poisson.mI.mB} follow from the first one, 
because $\{ \mI, \mB \}_k = \{ \mI, \mB_k \}$ for all $k=1,2,3$. 
\end{proof}

To conclude this section, we prove that 
the prime integrals $\mH, \mV, \mI, \mB_3$ have the following invariance property.

\begin{lemma}  \label{lemma:group.action}
The functionals $\mH, \mV, \mI, \mB_3$ 
defined in \eqref{def.mH}, \eqref{def.volume}, \eqref{def.mI}, \eqref{def.mB} 
(where $\mB_3$ is the third component of $\mB$) 
are invariant under the group action $\mT_\th$ 
defined in \eqref{def.mT.theta}, namely 
\begin{equation} \label{invariance.group.action}
\mH \circ \mT_\th = \mH, \quad \ 
\mV \circ \mT_\th = \mV, \quad \ 
\mI \circ \mT_\th = \mI, \quad \ 
\mB_3 \circ \mT_\th = \mB_3
\end{equation}
for all $\th \in \T := \R / 2 \pi \Z$. 
\end{lemma}

\begin{proof}
The first two identities are proved in \eqref{Ham.is.invariant}. 
For all $\th, \alpha \in \T$ one has
$\mT_\th \circ \mT_\alpha = \mT_{\alpha} \circ \mT_\th$.
Differentiating this identity with respect to $\alpha$ at $\alpha = 0$, 
and recalling $(iii)$ of Lemma \ref{lem:conservation}, one has 
\[
\mT_\th \circ \mM = \mM \circ \mT_\th
\]
for all $\th \in \T$. Thus the third identity in \eqref{invariance.group.action}
follows from the change of integration variable $R(\th) x = y$ in formula \eqref{def.mI}. 
As noticed in Lemma 6.2 of \cite{B.J.LM}, one has 
\begin{equation} \label{grad.mT.th}
\grad_{\S^2}( \mT_\th f) = R(\th)^T \mT_\th ( \grad_{\S^2} f ),
\end{equation}
where $R(\th)^T$ is the transpose 
of the matrix $R(\th)$ in \eqref{def.R(th)}. 
From \eqref{grad.mT.th}, with the change of integration variable $R(\th) x = y$ in formula \eqref{def.mB}, 
we deduce that 
\[
\mB \circ \mT_\th = R(\th)^T \mB.
\]
Taking the third component of this identity, by formula \eqref{def.R(th)}, 
we obtain the last identity in \eqref{invariance.group.action}.
\end{proof}

\section{The nonlinear operator}
\label{sec:nonlinear}
We define 
\begin{equation} \label{def.mF}
\mF(\om, u) := \grad \mH_{\sigma_0}(u) - \om \grad \mI(u) 
= \begin{pmatrix} 
\pa_\eta \mH_{\sigma_0}(\eta, \beta) - \om \pa_\eta \mI(\eta, \beta) \\ 
\pa_\beta \mH_{\sigma_0}(\eta, \beta) - \om \pa_\beta \mI(\eta, \beta) 
\end{pmatrix}
=: \begin{pmatrix} 
\mF_1(\om,u) \\ 
\mF_2(\om,u)
\end{pmatrix}, 
\end{equation}
where $u = (\eta, \beta)$, 
$\mH_{\s_0}$ is defined in \eqref{def.mH.sigma.0}, \eqref{def.mH}, \eqref{def.volume},   
and $\mI$ in \eqref{def.mI},
so that the critical point equation \eqref{eq.critical.point} is shortly written as 
\begin{equation} \label{eq.mF=0}
\mF(\om,u) = 0.
\end{equation}
By \eqref{quasi.Ham.syst.2.sigma.0} and \eqref{pa.mI}, 
the operators $\mF_1$ and $\mF_2$ defined in \eqref{def.mF} are
\begin{align}
\mF_1(\om,u) & = - (1 + \eta)^2 \big( X_2(\eta, \beta) + 2 \sigma_0 - \om \mM \beta \big),
\label{def.mF.1}
\\
\mF_2(\om,u) & = (1 + \eta)^2 \big( X_1(\eta, \beta) - \om \mM \eta \big),
\label{def.mF.2}
\end{align}
where $X_1, X_2$ are defined in \eqref{def.X.1}, \eqref{def.X.2}.
To take advantage of the results proved in \cite{B.J.LM}, 
where these operators are studied using a different order and sign convention 
for the two components of $\mF$, it is convenient to also define 
\begin{equation} \label{def.mF.old}
\mF_0(\om,u)
:= \begin{pmatrix} 
\mF_{0,1}(\om,u) \\ 
\mF_{0,2}(\om,u) 
\end{pmatrix}
:= \begin{pmatrix} 
\om \mM \eta - X_1(\eta, \beta) \\ 
\om \mM \beta - X_2(\eta, \beta)
\end{pmatrix},
\end{equation}
so that 
\[
\mF_1(\om,u) = (1+\eta)^2 (\mF_{0,2}(\om,u) - 2 \sigma_0), 
\quad \  
\mF_2(\om,u) = - (1+\eta)^2 \mF_{0,1}(\om,u), 
\]
that is,  
\begin{equation} \label{mF.mF.old}
\mF(\om,u) = (1+\eta)^2 J_0^{-1} \Big[ \mF_0(\om,u) - \begin{pmatrix} 0 \\ 2 \sigma_0 \end{pmatrix} \Big],
\quad \ 
J_0 := \begin{pmatrix}
0 & -1 \\ 
1 & 0 
\end{pmatrix}, 
\quad 
J_0^{-1} = - J_0.
\end{equation}
We underline that $\mF_0$ is the operator studied in \cite{B.J.LM}. 
Note that the order and sign convention of $\mF$ is motivated 
by the variational structure of the problem, 
as $\mF$ in \eqref{def.mF} is the gradient of a functional,
while the order and sign convention of $\mF_0$ 
is that of system \eqref{syst.mM.eta.mM.beta}. 

The next lemma concerns the regularity of $\mF$.  

\begin{lemma} 
\label{lemma:mF.Hs}
Let $s_0, s \in \R$, $s \geq s_0 > 1$. 
There exists $\delta_0 > 0$, depending on $s_0$, with the following property. 
Let
\[
U := \{ u = (\eta, \beta) : 
\eta \in H^{s+\frac32}(\S^2,\R), \  
\beta \in H^{s+1}(\S^2,\R), \  
\| \eta \|_{H^{s_0 + \frac32}(\S^2)} < \delta_0 \}.
\]
Then 
$\mF_1 (\om,u) \in H^{s-\frac12}(\S^2,\R)$,
$\mF_2(\om,u) \in H^s(\S^2,\R)$  
for all $(\om,u) \in \R \times U$, 
and the map 
\[
\mF : \R \times U \to H^{s-\frac12}(\S^2,\R) \times H^{s}(\S^2,\R)
\]
is analytic.
\end{lemma}

\begin{proof}
The lemma follows from Lemma 6.3 of \cite{B.J.LM}, where the corresponding statement is proved 
for $\mF_0$, and from identity \eqref{mF.mF.old} which gives $\mF$ in terms of $\mF_0$. 
\end{proof}

In the next lemma we use the properties of the conserved quantities $\mH, \mI,  \mV, \mB$ 
of Section \ref{sec:symmetries} to obtain important orthogonality properties 
for the image of the nonlinear operator $\mF$.

\begin{lemma} \label{lemma:orthogonality.mF}
Assume the hypotheses of Lemma \ref{lemma:mF.Hs}.
Then, for all $(\om, u) \in \R \times U$, 
one has the orthogonality identities 
\begin{align} 
& \langle \mF(\om, u) \,,\, (1+\eta)^{-2} J_0^{-1} \nabla \mV(u) \rangle_{L^2(\S^2)} = 0,
\label{eq:orthogonal.mV}
\\
& \langle \mF(\om, u) \,,\, (1+\eta)^{-2} J_0^{-1} \nabla \mB_3(u) \rangle_{L^2(\S^2)} = 0,
\label{eq:orthogonal}
\end{align}
where $\mB_3$ is the third component of $\mB$ in \eqref{def.mB},
and $\mV$ is in \eqref{def.volume}.
\end{lemma}

\begin{proof}
By definition, $\mF = \nabla \mH_{\sigma_0} - \omega \nabla \mI$
and $\mH_{\sigma_0} = \mH - 2 \sigma_0 \mV$, 
see \eqref{def.mF} and \eqref{def.mH.sigma.0}. 
By $(iii)$ in Lemma \ref{lemma:mass}, 
$(ii)$ in Lemma \ref{lem.invariance}, 
and the last identity in Lemma \ref{lem:poisson.Mi.MB}, 
the Poisson brackets
\[
\{ \mH , \mB_3 \}, \quad 
\{ \mV , \mB_3 \}, \quad 
\{ \mI , \mB_3 \}, \quad 
\{ \mH , \mV \}, \quad 
\{ \mV , \mV \}, \quad 
\{ \mI , \mV \}
\]
all vanish. 
Then \eqref{eq:orthogonal.mV} and \eqref{eq:orthogonal} are an immediate consequence 
of definition \eqref{Poisson.brackets}:
\[
\langle \mF(\om, u) \,,\, (1+\eta)^{-2} J_0^{-1} \nabla \mB_3(u) \rangle_{L^2(\S^2)}
= \{ \mH, \mB_3 \}(u) - 2 \sigma_0 \{ \mV , \mB_3 \}(u) - \omega \{ \mI, \mB_3 \}(u) = 0,
\]
and similarly for $\mV$. 
\end{proof}

Let us emphasize the meaning of Lemma \ref{lemma:orthogonality.mF}.
By \eqref{pa.mV}, \eqref{pa.mB}, and the definition of $J_0^{-1}$ in \eqref{mF.mF.old}, 
one has 
\begin{equation} \label{orthogonal.vectors}
(1+\eta)^{-2} J_0^{-1} \nabla \mV(u) 
= \begin{pmatrix} 0 \\ -1 \end{pmatrix}, 
\quad \ 
(1+\eta)^{-2} J_0^{-1} \nabla \mB_3(u) 
= \begin{pmatrix} x_3 - (1+\eta)^{-1} (\grad_{\S^2} \eta)_3 \\ 
- (1+\eta)^{-1} (\grad_{\S^2} \beta)_3 \end{pmatrix}.  
\end{equation}
Thus, identities \eqref{eq:orthogonal.mV}, \eqref{eq:orthogonal} guarantee 
the $L^2(\S^2)$ orthogonality 
\begin{equation}\label{eq:orthogonal.vectors}
\mF(\om, u) \perp \begin{pmatrix} 0 \\ 1 \end{pmatrix}, 
\quad \ 
\mF(\om, u) \perp \, \begin{pmatrix} x_3 \\ 0 \end{pmatrix}
- \frac{1}{1+\eta} \begin{pmatrix}
(\nabla_{\S^2} \eta)_3 \\
(\nabla_{\S^2} \beta)_3 \end{pmatrix},
\end{equation}
meaning that $\mF(\om, u)$ is orthogonal to the pair $(0, 1)$ 
and to a perturbation of the pair $(x_3, 0)$, 
for $u = (\eta, \beta)$ in a neighborhood of the origin.
As we will see, this is a key observation in order to prove our main result.

We conclude this section by observing the following conjugation property. 

\begin{lemma}  \label{lemma:conj.mF}
For all $\th \in \T$, $(\om,u) \in \R \times U$, one has  
\begin{alignat}{2}
(\grad \mH) \circ \mT_\th & =  \mT_\th \circ (\grad \mH), 
\quad \ &
(\grad \mV) \circ \mT_\th & =  \mT_\th \circ (\grad \mV), 
\quad \ 
(\grad \mI) \circ \mT_\th =  \mT_\th \circ (\grad \mI), 
\notag \\ 
(\grad \mB_3) \circ \mT_\th & =  \mT_\th \circ (\grad \mB_3),
\quad \ & 
\mF( \om, \mT_\th u ) & = \mT_\th \mF( \om, u ).
\label{conj.mF}
\end{alignat}
\end{lemma}

\begin{proof}
By Lemma \ref{lemma:group.action}, $\mH (\mT_\th u) = \mH(u)$. 
Taking the derivative with respect to $u$ in direction $\tilde u$ gives 
$\mH'(\mT_\th u)[\mT_\th \tilde u] = \mH'(u)[\tilde u]$, i.e., 
\[
\la (\grad \mH)(\mT_\th u) , \, \mT_\th \tilde u \ra_{L^2(\S^2)} 
= \la \grad \mH(u) , \, \tilde u \ra_{L^2(\S^2)}.
\]
Hence 
\[
\mT_\th^{\,T} (\grad \mH)(\mT_\th u) = \grad \mH(u),
\]
where $\mT_\th^{\,T}$ is the transpose of $\mT_\th$ with respect to the $L^2(\S^2)$ scalar product. 
One can easily notice that $\mT_\th^{\,T} = \mT_{- \th} = \mT_\th^{-1}$, 
so that the first identity of the lemma is proved. 
The other identities are proved similarly; 
recall that $\mF(\om, u) = \grad \mH(u) - 2 \sigma_0 \grad \mV(u) - \om \grad \mI(u)$
by \eqref{def.mF} and \eqref{def.mH.sigma.0}. 
\end{proof}

\section{The linearized operator}\label{sec:linearized}

To study the nonlinear equation \eqref{eq.mF=0}, 
we consider its linearized operator at $u=0$
\[
L_\om := \pa_u \mF(\om,0) : H^{s+\frac32}(\S^2) \times H^{s+1}(\S^2) 
\to H^{s - \frac12}(\S^2) \times H^s(\S^2),
\]
\begin{equation} \label{def.L}
L_\om (\eta, \beta)  
= \begin{pmatrix} 
- \s_0 (2 + \Delta_{\S^2}) & \om \pois 
\\
-\om \pois & G(0) 
\end{pmatrix}
\begin{pmatrix}
\eta \\ 
\beta
\end{pmatrix},
\end{equation}
and we recall that the bifurcation from zero of nontrivial solutions of \eqref{eq.mF=0} 
can only occur from values of $\om$ for which $L_\om$ has a nontrivial kernel. 
Owing to \eqref{mF.mF.old}, one has 
\[
L_\om = J_0^{-1} L_{\mF_0, \om}, 
\quad \ 
L_{\mF_0, \om} := \pa_u \mF_0(\om,0),
\]
where $J_0^{-1}$ is defined in \eqref{mF.mF.old}. 
The operator $L_{\mF_0,\om}$ is studied in \cite{B.J.LM} 
thanks to the simple identities 
\begin{equation}  \label{eigenv.rule}
- (2 + \Delta_{\S^2}) \ph_{\ell,m} = (\ell + 2)(\ell - 1) \ph_{\ell,m}, \quad 
G(0) \ph_{\ell,m} = \ell \ph_{\ell,m}, \quad  
\mM \ph_{\ell,m} = - m \ph_{\ell,-m},
\end{equation}
which hold for all $(\ell, m)\in T$, 
where $\ph_{\ell,m}$ and $T$ are given in \eqref{def.ph.ell.m.intro}, \eqref{def.ph.ell.m.Re.Im}.
In the next proposition we collect some properties of $L_{\mF_0,\om}$ 
proved in \cite[Section 6]{B.J.LM}.

\begin{proposition}[From \cite{B.J.LM}]
\label{prop:L0}
Let $(\ell_0, m_0) \in T$ with $\ell_0 \geq 2$ and $m_0 \neq 0$.  
Let 
\begin{equation} \label{om.fix}
\om_0 := \sqrt{\sigma_0} \frac{ \sqrt{(\ell_0+2)(\ell_0-1)\ell_0} }{m_0}
\end{equation}
and 
\begin{equation} \label{def.S}
S:=\{ (\ell, m) \in T :  (\ell+2)(\ell-1)\ell = c_0 m^2 \}, \quad \ 
c_0 := (\ell_0+2)(\ell_0-1) \ell_0 m_0^{-2}.
\end{equation}
Then the set $S$ has a finite number of elements, which are $(0,0)$, $(1,0)$, 
$(\ell_0, \pm m_0)$, and possibly finitely many other pairs $(\ell, \pm m)$ 
with $\ell \geq 2$ and $1 \leq |m| \leq \ell$.  
The kernel of the linear operator $L_{\mF_0,\om_0}$ is the finite dimensional space 
\begin{equation} \label{def.V}
V := \ker L_{\mF_0,\om_0} = \bigg\{ 
\begin{pmatrix} \eta \\ \beta \end{pmatrix} 
= \lm_{0,0} \begin{pmatrix} 0 \\ \ph_{0,0} \end{pmatrix} 
+ \sum_{ \begin{subarray}{c} (\ell,m) \in S \\ \ell \geq 1 \end{subarray}} 
\lm_{\ell,m} \begin{pmatrix} \ell \ph_{\ell,m} \\  - \om_0 m \ph_{\ell,-m} \end{pmatrix}
: \lm_{\ell,m} \in \R \bigg\}.
\end{equation}
Its orthogonal complement in $L^2(\S^2) \times L^2(\S^2)$ is the vector space
\begin{multline*}
W:= \bigg\{ 
\begin{pmatrix} \eta \\ \beta \end{pmatrix} 
= \lm_{0,0} \begin{pmatrix} \ph_{0,0} \\ 0 \end{pmatrix} 
+ \sum_{ \begin{subarray}{c} (\ell,m) \in S \\ \ell \geq 1 \end{subarray}}
\lm_{\ell,m} \begin{pmatrix} \om_0 m \ph_{\ell,m} \\ \ell \ph_{\ell,-m} \end{pmatrix}
+ \sum_{(\ell,m) \in T \setminus S} 
\begin{pmatrix} \hat \eta_{\ell,m} \ph_{\ell,m} \\  \hat \beta_{\ell,m} \ph_{\ell,m} \end{pmatrix}
\\ 
: \lm_{\ell,m}, \hat \eta_{\ell,m}, \hat \beta_{\ell,m}  \in \R, \  
(\eta, \beta) \in L^2(\S^2) \times L^2(\S^2) \bigg\}.
\end{multline*} 
The range of $L_{\mF_0,\om_0}$ is contained in
\begin{multline*} 
R_{\mF_0} := \bigg\{ 
\begin{pmatrix} f \\ g \end{pmatrix} 
= \hat g_{0,0} \begin{pmatrix} 0 \\ \ph_{0,0} \end{pmatrix} 
+ \sum_{ \begin{subarray}{c} (\ell,m) \in S \\ \ell \geq 1 \end{subarray}}
\hat f_{\ell,-m} \begin{pmatrix} \ph_{\ell,-m} \\ -\om_0 m \ell^{-1} \ph_{\ell,m} \end{pmatrix}
+ \sum_{(\ell,m) \in T \setminus S} 
\begin{pmatrix} \hat f_{\ell,m} \ph_{\ell,m} \\ \hat g_{\ell,m} \ph_{\ell,m} \end{pmatrix}  
\\ 
: \hat f_{\ell,m}, \hat g_{\ell,m} \in \R, \ 
(f,g) \in L^2(\S^2) \times L^2(\S^2) \bigg\}.
\end{multline*}
The orthogonal complement of $R_{\mF_0}$ with respect to the scalar product of 
$L^2(\S^2) \times L^2(\S^2)$ is the finite-dimensional space 
\begin{equation} \label{def.Z} 
Z_{\mF_0} := \bigg\{ 
\begin{pmatrix} f \\ g \end{pmatrix} 
= \lm_{0,0} \begin{pmatrix} \ph_{0,0} \\ 0 \end{pmatrix} 
+ \sum_{ \begin{subarray}{c} (\ell,m) \in S \\ \ell \geq 1 \end{subarray}} 
\lm_{\ell,m} \begin{pmatrix} \om_0 m \ph_{\ell,-m} \\ \ell \ph_{\ell,m} \end{pmatrix}
: \lm_{\ell,m} \in \R \bigg\}.
\end{equation}
Define
\begin{equation} \label{def.Ws.R0s}
W^s := W \cap (H^{s+\frac32}(\S^2)\times H^{s+1} (\S^2)), 
\quad \ 
R_{\mF_0}^s := R_{\mF_0} \cap (H^s (\S^2)\times H^{s-\frac12}(\S^2)  ).
\end{equation}
The linear map $L_{\mF_0,\om_0}|_{W^s} : W^s \to R_{\mF_0}^s$ is bijective. 
Its inverse $(L_{\mF_0,\om_0}|_{W^s})^{-1} : R_{\mF_0}^s \to W^s$ is bounded, with
\begin{equation} \label{inv.est.L}
\| (L_{\mF_0,\om_0}|_{W^s})^{-1} (f,g) \|_{H^{s+\frac32}(\S^2) \times H^{s+1}(\S^2)} 
\leq C_s \| (f,g) \|_{H^{s}(\S^2) \times H^{s-\frac12}(\S^2)} 
\end{equation}
for all $(f,g) \in R_0^s$. 
The constant $C_s$ depends on $\sigma_0, \ell_0, s$.  
\end{proposition} 

We are interested in the operator $L_{\om_0} = J_0^{-1} L_{\mF_0,\om_0}$. 
Observe that $J_0^{-1}$ acts as a rotation on the arrival space of the operator, leaving the domain unchanged.
Thus, 
\[
\operatorname{ker}(L_{\om_0}) = \operatorname {ker}(L_{\mF_0,\om_0})=V.
\]
The range of $L_{\om_0}$ satisfies 
\[
R := \operatorname{Range}(L_{\om_0}) = J_0^{-1} \operatorname{Range}(L_{\mF_0,\om_0}),
\]
whence
\begin{multline} 
R = \bigg\{ 
\begin{pmatrix} f \\ g \end{pmatrix} 
= \hat f_{0,0} \begin{pmatrix} \ph_{0,0} \\ 0 \end{pmatrix} 
+ \sum_{ \begin{subarray}{c} (\ell,m) \in S \\ \ell \geq 1 \end{subarray}}
\hat f_{\ell,-m} \begin{pmatrix} \om_0  m \ph_{\ell,m} \\ \ell \ph_{\ell,-m} \end{pmatrix}
+ \sum_{(\ell,m) \in T \setminus S} 
\begin{pmatrix} \hat f_{\ell,m} \ph_{\ell,m} \\ \hat g_{\ell,m} \ph_{\ell,m} \end{pmatrix}  
\\ 
: \hat f_{\ell,m}, \hat g_{\ell,m} \in \R, \ 
(f,g) \in L^2(\S^2) \times L^2(\S^2) \bigg\}.
\label{def.R}
\end{multline}
We also define 
\[
R^s := J_0^{-1} R_{\mF_0}^s
= R \cap (H^{s-\frac12}(\S^2) \times H^s (\S^2)), 
\quad \ 
Z := J_0^{-1} Z_{\mF_0}.
\]
Thus $Z$ is the orthogonal complement of $R$ in $L^2(\S^2) \times L^2(\S^2)$. 
We also note that 
\begin{equation} \label{Z.is.V.R.is.W}
Z = V, \quad \ 
R = W,
\end{equation}
while, for any real $s, \sigma$, the spaces $R^s$ and $W^\s$ are not equal 
(if $R^s = W^\sigma$, then both $s - \frac12 = \sigma + \frac32$ and $s = \sigma + 1$, 
which is impossible);  
on the contrary, $Z = V$ because they are finite-dimensional vector spaces,  
and $R = W$ as subspaces of $L^2(\S^2) \times L^2(\S^2)$. 
Since $V$ is finite-dimensional, we denote 
\begin{equation} \label{def.norma.V}
|v| := \| v \|_{L^2(\S^2) \times L^2(\S^2)}
\end{equation}
for all $v \in V$. 
Also, we introduce the notation 
\begin{equation} \label{def.norma.Ws.Rs}
\| (\eta, \beta) \|_{W^s} := \| \eta \|_{H^{s+\frac32}(\S^2)} + \| \beta \|_{H^{s+1}(\S^2)}, 
\quad \ 
\| (f,g) \|_{R^s} := \| f \|_{H^{s-\frac12}(\S^2)} + \| g \|_{H^s(\S^2)}, 
\end{equation}
for all $(\eta, \beta) \in W^s$, all $(f,g) \in R^s$.
We denote by 
\[
\Pi_R \colon R \oplus Z \to R, 
\quad \ 
\Pi_Z \colon R \oplus Z \to Z
\]
the projection maps onto $R$ and $Z$, 
and we also write $\Pi_V = \Pi_Z$, $\Pi_W = \Pi_R$.
To keep track of the invariance property under the group action $\mT_\th$, 
we make the following observations. 
Recall that $T$ is the set of indices defined in \eqref{def.ph.ell.m.intro}.

\begin{lemma}  \label{lemma:V.W.group.action}
For all $(\ell, m) \in T$ with $m \geq 0$, 
for all $\th \in \T$, one has 
\begin{align}
\mT_\th \ph_{\ell,m} & = \cos(m \th) \ph_{\ell,m} - \sin(m \th) \ph_{\ell,-m}, 
\notag \\ 
\mT_\th \ph_{\ell,-m} & = \sin(m \th) \ph_{\ell,m} + \cos(m \th) \ph_{\ell,-m}.
\label{mT.th.ph.ell.m}
\end{align}
As a consequence, for all $(\ell, m) \in T$, 
the linear space generated by $\ph_{\ell,m}, \ph_{\ell, - m}$ is invariant for $\mT_\th$. 
Moreover, $V$ and $W$ are also invariant for $\mT_\th$, namely 
\begin{equation} \label{mT.th.proj}
\mT_\th(V) = V, \quad \ 
\mT_\th(W) = W, \quad \ 
\mT_\th \Pi_V = \Pi_V \mT_\th, \quad \ 
\mT_\th \Pi_W = \Pi_W \mT_\th.
\end{equation}
\end{lemma}

\begin{proof}
Let $x \in \R^3$ and $y := R(\th) x$, where $R(\th)$ is in \eqref{def.R(th)}. 
Then the complex number $y_1 + i y_2$ is equal to the product $e^{i \th} (x_1 + i x_2)$, 
while $y_3 = x_3$. Therefore, for any integer $m \geq 0$, one has 
\[
(y_1 + i y_2)^m = e^{i m \th} (x_1 + i x_2)^m
\]
and, by \eqref{def.ph.ell.m.Re.Im} and \eqref{def.mT.theta}, 
\begin{align*}
\mT_\th \ph_{\ell,m}(x) 
& = \ph_{\ell,m} (R(\th) x) 
\\ 
& = \ph_{\ell,m} (y) 
\\
& = c_{\ell}^{(m)} P_\ell^{(m)}(y_3) \Re[ (y_1 + i y_2)^m ]
\\ 
& = c_{\ell}^{(m)} P_\ell^{(m)}(x_3) \Re[ e^{i m \th} (x_1 + i x_2)^m ]
\\ 
& = c_{\ell}^{(m)} P_\ell^{(m)}(x_3) \{ \cos(m\th) \Re[ (x_1 + i x_2)^m ]
- \sin(m\th) \Im[ (x_1 + i x_2)^m ] \},
\end{align*}
which is the first identity in \eqref{mT.th.ph.ell.m}. 
The second identity in \eqref{mT.th.ph.ell.m} is proved similarly. 
By \eqref{mT.th.ph.ell.m}, for $(\ell,m) \in T$, with $m \geq 0$, one has 
\begin{align}
\mT_\th \begin{pmatrix} 
\ell \ph_{\ell,m} \\ 
- \om_0 m \ph_{\ell, -m} 
\end{pmatrix}
& = \cos(m \th) \begin{pmatrix} 
\ell \ph_{\ell,m} \\ 
- \om_0 m \ph_{\ell,-m} 
\end{pmatrix}
- \sin(m \th)
\begin{pmatrix} 
\ell \ph_{\ell,-m} \\ 
\om_0 m \ph_{\ell,m} 
\end{pmatrix},
\notag \\
\mT_\th \begin{pmatrix} 
\ell \ph_{\ell,-m} \\ 
\om_0 m \ph_{\ell, m} 
\end{pmatrix}
& = \cos(m \th) \begin{pmatrix} 
\ell \ph_{\ell,-m} \\ 
\om_0 m \ph_{\ell,m} 
\end{pmatrix}
+ \sin(m \th)
\begin{pmatrix} 
\ell \ph_{\ell,m} \\ 
- \om_0 m \ph_{\ell,-m} 
\end{pmatrix}.
\label{mT.th.ttv}
\end{align}
Hence, recalling the definition \eqref{def.V} of $V$, 
we deduce that $\mT_\th v \in V$ for all $v \in V$. 
This means that $\mT_\th(V) \subseteq V$ for all $\th \in \T$. 
Applying $\mT_{- \th}$ to this inclusion, we get $V \subseteq \mT_{-\th}(V)$
for all $\th \in \T$, whence $V = \mT_\th(V)$ for all $\th$. 
Now let $w \in W$, and consider $\mT_\th w$, which is the sum $\mT_\th w = v_1 + w_1$ 
for some $v_1 \in V$, $w_1 \in W$. One has 
\[
\la \mT_\th w , v_1 \ra_{L^2(\S^2)} 
= \la w , \mT_\th^{\,T} v_1 \ra_{L^2(\S^2)} 
= \la w , \mT_{-\th} v_1 \ra_{L^2(\S^2)} 
= 0
\]
because $\mT_{-\th} v_1 \in V$ and $w \in W$. 
On the other hand, one also has 
\[
\la \mT_\th w , v_1 \ra_{L^2(\S^2)} 
= \la v_1 + w_1 , v_1 \ra_{L^2(\S^2)} 
= |v_1|^2,
\]
whence $v_1 = 0$. Hence $\mT_\th w = w_1 \in W$. 
This proves that $\mT_{\th}(W) \subseteq W$ for all $\th \in \T$. 
Applying $\mT_{-\th}$ to this inclusion gives $W \subseteq \mT_{-\th}(W)$ 
for all $\th \in \T$, whence $\mT_\th (W) = W$. 
\end{proof}

\begin{remark}
\label{rem:why.not.using.mF.old}
Concerning the definitions \eqref{def.mF} and \eqref{def.mF.old}, 
one could wonder why introducing $\mF$ should be more convenient 
than just keeping $\mF_0$, which is already studied in \cite{B.J.LM}.
The reason is that, with $\mF_0$, the matrix $(L_0)_{(\ell,m)}$ is a Jordan block for $(\ell,m) = (0,0)$ 
and $(1,0)$, while using $\mF$ the corresponding matrices are diagonal, with one zero eigenvalue. 
Jordan blocks create troubles proving that the bifurcation equation 
is still a variational problem, but this difficulty is 
easy to overcome, as it can be removed by 
$J_0$ in \eqref{def.mF.old}.
This problem is not visible in \cite{B.J.LM}, 
both because there the variational structure is not explicitly used, 
and because the frequencies $(0,0)$ and $(1,0)$ are removed by considering 
invariant subspaces of even/odd functions. 
On the contrary, Theorem \ref{thm:main} does not rely on any symmetry assumption for the unknown functions.
\end{remark}

\section{Lyapunov-Schmidt decomposition and solution of the range equation}\label{sec:LS.decomposition}

We perform the Lyapunov-Schmidt decomposition. 
Decomposing any $u \in H^{s+\frac32}(\S^2)\times H^{s+1}(\S^2)$ 
as the sum $u=v+w$ where $v\in V$ and $w\in W^s$, 
the equation $\mF (\om, u)=0$ can be reformulated as
\begin{align}
&\Pi_{R^s} \mathcal{F}(\omega , v+w) =0,\label{range.equation} \\
&\Pi_{Z} \mathcal{F}(\omega , v+w) =0.\label{Bifurcation.equation}
\end{align}
Here \eqref{range.equation} is the range equation, 
while \eqref{Bifurcation.equation} is the bifurcation equation. 

For any $\e > 0$, recalling the notation in \eqref{def.norma.V}, denote 
\begin{equation} \label{def.mU.e}
\mU_\e := \{ (\om, v) \in \R \times V : |\om - \om_0| < \e, \ |v| < \e \}.
\end{equation}
In the next lemma we observe that the range equation can be solved 
by the implicit function theorem.

\begin{lemma}[Solution of the range equation] 
\label{lemma:range.eq} 
There exist $\e_0 > 0$, $C > 0$, and an analytic function $\mU_{\e_0} \to W^s$, 
$(\om, v) \mapsto w(\om,v)$ such that 
\begin{equation} \label{IFT.w}
\Pi_{R^s} \mF ( \om, v + w(\om,v)) = 0
\end{equation}
for all $(\om,v) \in \mU_{\e_0}$, and the only solution $(\om,v,w)$ 
of the range equation \eqref{range.equation} 
with $(\om , v) \in \mU_{\e_0}$ and $\| w \|_{W^s} < C \e_0^2$ 
is the triplet $(\om, v, w(\om,v))$.
Moreover, for all $(\om, v) \in \mU_{\e_0}$, all $\th \in \T$, one has 
\begin{align} \label{w.pav.w}
w(\om,0) & = 0 \quad \forall |\om-\om_0| < \e_0, 
\qquad 
\pa_v w(\om_0, 0) = 0, 
\\
\label{w.estimate}
\| w(\om,v) \|_{W^s} 
& \leq C ( |v|^2 + |v| |\om-\om_0| ),
\\
\label{w.group.action}
w(\om,\mT_\th v) 
& = \mT_\th w(\om, v).
\end{align}
\end{lemma}

\begin{proof} 
The operator in the left-hand side of \eqref{range.equation} is analytic in $\om,v,w$ 
(see Lemma \ref{lemma:mF.Hs}), and its linearization with respect to $w$ 
at the point $(\om,v,w) = (\om_0,0,0)$ is the operator $L_{\om_0}$, 
which is a linear homeomorphism of $W^s$ onto $R^s$.  
Hence the existence and uniqueness in the statement comes from the implicit function theorem 
for analytic operators in Banach spaces.
The first identity in \eqref{w.pav.w} holds by the uniqueness of the solution 
and because $\mF(\om,0) = 0$ for all $\om$.  
The derivative of identity \eqref{IFT.w} with respect to $v$ 
in any direction $\tilde v \in V$ at $(\om,v) = (\om_0,0)$ gives 
\[
0 = \Pi_{R^s} (\pa_u \mF)(\om_0,0) \big[ \tilde v + \pa_v w(\om_0,0)[\tilde v] \big] 
= \Pi_{R^s} L_{\om_0}|_{W^s} \big[ \pa_v w(\om_0,0)[\tilde v] \big],
\]
which implies the second identity in \eqref{w.pav.w}. 
Then \eqref{w.estimate} follows by \eqref{w.pav.w} and the Taylor expansion of $w(\om,v)$ 
around $(\om_0,0)$.
Applying $\mT_\th$ to identity \eqref{IFT.w}, 
and using the third identity in \eqref{mT.th.proj}
and the last one in \eqref{conj.mF},  
one has 
\[
0 = \mT_\th \Pi_{R^s} \mF ( \om, v + w(\om,v))
= \Pi_{R^s} \mT_\th \mF ( \om, v + w(\om,v))
= \Pi_{R^s} \mF ( \om, \mT_\th v + \mT_\th w(\om,v)).
\]
On the other hand, \eqref{IFT.w} applied with $\mT_\th v$ in place of $v$ gives 
\[
\Pi_{R^s} \mF ( \om, \mT_\th v + w(\om, \mT_\th v)) = 0. 
\]
Thus \eqref{w.group.action} follows from the uniqueness property of the implicit function. 
\end{proof}

\section{Degenerate and non-degenerate decomposition}
\label{sec:degenerate.decomposition}
For notation convenience, recalling \eqref{def.V}, we define 
\begin{equation} \label{def.ttv}
\mathtt v_{0,0} := \begin{pmatrix} 0 \\ \ph_{0,0} \end{pmatrix}, 
\quad \ 
\mathtt v_{\ell,m} := \frac{ 1 }{ \sqrt{\ell^2 + \om_0^2 m^2} }
\begin{pmatrix} \ell \ph_{\ell,m} \\ - \om_0 m \ph_{\ell,-m} \end{pmatrix},
\quad \ (\ell,m) \in S, \ \ell \geq 1.
\end{equation}
Since each $\ph_{\ell,m}$ has unitary $L^2(\S^2,\R)$ norm, one has 
\begin{equation}  \label{norm.ttv}
|\mathtt v_{\ell,m}| = 1 \quad \forall (\ell,m) \in S, 
\end{equation}
where notation \eqref{def.norma.V} is used. 
Thus the vectors in \eqref{def.ttv} form an orthonormal basis of $V$ 
with respect to the $L^2(\S^2,\R) \times L^2(\S^2,\R)$ scalar product. 
We split the set $S$ defined in \eqref{def.S} 
into its ``degenerate'' and ``non-degenerate'' part, 
\begin{equation} \label{def.SD.SN}
S_D := \{ (0,0), (1,0) \}, \quad \ 
S_N := \{ (\ell,m) \in S : \ell > 1 \},
\end{equation}
and we decompose $V$ into the corresponding linear subspaces
\begin{equation} \label{def.VD.VN}
V_D := \mathrm{span}_\R \{ \mathtt v_{(\ell,m)} : (\ell,m) \in S_D \}, \quad \ 
V_N := \mathrm{span}_\R \{ \mathtt v_{(\ell,m)} : (\ell,m) \in S_N \}.
\end{equation}
Thus, $S = S_D \cup S_N$ and $V = V_D \oplus V_N$.
Since $Z=V$, we also define $Z_D := V_D$, $Z_N := V_N$. 
From Lemma \ref{lemma:V.W.group.action} it follows that 
$V_D$ and $V_N$ are invariant spaces for $\mT_\th$; more precisely, one has
\begin{equation}  \label{VD.VN.group.action}
\mT_\th |_{V_D} = \mathrm{id}, \quad \ 
\mT_\th (V_N) = V_N.
\end{equation}

By \eqref{mT.th.ttv} and \eqref{def.ttv}, 
one has 
\begin{equation} \label{mT.th.ttv.bis}
\mT_\th \mathtt{v}_{\ell,m} 
= \cos(m \th) \mathtt{v}_{\ell,m} 
- \sin(m \th) \mathtt{v}_{\ell,-m}
\end{equation}
for all $(\ell,m) \in S_N$; note that \eqref{mT.th.ttv.bis} 
holds not only for $m$ positive, but also for $m$ negative. 
By \eqref{mT.th.ttv.bis}, 
for any $v = \sum_{(\ell,m) \in S_N} \hat v_{\ell,m} \mathtt{v}_{\ell,m} \in V_N$ one has 
\begin{align}
\mT_\th v 
& = \sum_{(\ell,m) \in S_N} \hat v_{\ell,m} \cos(m \th) \mathtt{v}_{\ell,m} 
- \sum_{(\ell,m) \in S_N} \hat v_{\ell,m} \sin(m \th) \mathtt{v}_{\ell,-m}
\notag \\
& = \sum_{(\ell,m) \in S_N} \hat v_{\ell,m} \cos(m \th) \mathtt{v}_{\ell,m} 
- \sum_{(\ell,m) \in S_N} \hat v_{\ell,-m} \sin(-m \th) \mathtt{v}_{\ell,m}
\notag \\
& = \sum_{(\ell,m) \in S_N} \{ \cos(m \th) \hat v_{\ell,m} 
+ \sin(m \th) \hat v_{\ell,-m} \} \mathtt{v}_{\ell,m},
\label{mT.th.v}
\end{align}
where we have made a change of summation variable in the second sum. 
In other words, if $S_N = \{ (\ell_1, m_1), (\ell_1, - m_1), \ldots, (\ell_n, m_n), (\ell_n, - m_n) \}$,  
and $\hat v = (\hat v_{\ell_1, m_1} , \ldots, \hat v_{\ell_n, - m_n}) \in \R^{2n}$ 
is the vector of the coefficients of $v$ with respect to the basis 
$\{ \mathtt v_{\ell_1, m_1} , \ldots, \mathtt v_{\ell_n, - m_n} \}$ of $V_N$, 
then the linear map $\mT_\th : V_N \to V_N$ 
is represented by the block-diagonal matrix $M = \mathrm{diag}(M_1, \ldots, M_n)$, where 
the $j$-th block is 
\[
M_j = \begin{pmatrix} 
\cos(m_j \th) & \sin(m_j \th) \\ 
- \sin(m_j \th) & \cos(m_j \th) 
\end{pmatrix}, 
\quad \ j = 1, \ldots, n.
\]

\section{Choice of $\om(v)$}
\label{sec:choice.of.omega}

In this section we determine $\om$ as a function of $v$, 
see Lemma \ref{lemma:choice.of.omega}.
Let us consider the scalar product 
\begin{align} \label{def.F}
F(\om,v) & := \la f(\om, v) , \, g(\om,v) \ra_{L^2(\S^2)},
\\
\label{def.f.g}
f(\om, v) & := \Pi_{Z_N} \mF(\om, v + w(\om, v)), \quad \ 
g(\om, v) := (\grad \mI)(v + w(\om,v)).
\end{align}

\begin{lemma} \label{lemma:F.f.g.analytic}
The functions 
$f : \mU_{\e_0} \to Z_N$, 
$g : \mU_{\e_0} \to L^2(\S^2,\R^2)$, 
and $F : \mU_{\e_0} \to \R$ 
defined in \eqref{def.F}, \eqref{def.f.g} 
are analytic in the open neighborhood $\mU_{\e_0}$ 
of $(\om,v) = (\om_0, 0)$ given by Lemma \ref{lemma:range.eq}. 
\end{lemma}

\begin{proof} 
It trivially follows from \eqref{pa.mI} and Lemmas \ref{lemma:mF.Hs}, \ref{lemma:range.eq}.
\end{proof}

Our goal is to obtain a function $\om = \om(v)$ such that $F(\om(v),v)=0$ around $(\om,v) = (\om_0, 0)$
by the implicit function theorem. 
To this purpose, we calculate the derivatives of $F$ at $(\om_0,0)$. 
Since $v \mapsto F(\om,v)$ is an analytic function starting with quadratic terms, 
we have to consider second order derivatives of $F$ with respect to $v$.

\begin{lemma} 
One has 
\begin{align}
f(\om, 0) & = 0, 
\quad 
\pa_v f(\om_0,0) = 0, 
\quad 
\pa_{\om v} f(\om_0,0) 
= - \Pi_{Z_N} J_0 \mM,
\label{der.f}
\\
g(\om, 0) & = 0, 
\quad 
\pa_v g(\om_0,0) 
= d (\grad \mI)(0) 
= J_0 \mM
\label{der.g}
\end{align}
for any $|\om - \om_0| \leq \e_0$, 
where $J_0$ is defined in \eqref{mF.mF.old}, 
and, recalling the notation \eqref{def.norma.V}, 
\begin{align}
& F(\om, 0) = 0, 
\quad \ 
\pa_v F(\om, 0) = 0, 
\quad \ 
\pa_{vv} F(\om_0,0) = 0, 
\notag 
\\ 
& \pa_{\om v v} F(\om_0,0) [\tilde v, \tilde v] 
= - | \Pi_{Z_N} J_0 \mM \tilde v |^2 
\quad \forall \tilde v \in V.
\label{der.F}
\end{align}
\end{lemma}

\begin{proof}
Recall that $w(\om,0) = 0$, see \eqref{w.pav.w}. 
Hence $g(\om,0) = \grad \mI(0) = 0$ by \eqref{pa.mI}. 
The second identity in \eqref{der.g} follows from \eqref{pa.mI}.
One has $f(\om,0) = 0$ because $\mF(\om,0) = 0$ for all $\om$.  
Also, 
\[
\pa_v f(\om,v)[\tilde v] 
= \Pi_{Z_N} (\pa_u \mF)(\om, v + w(\om,v)) [ \tilde v + \pa_v w(\om,v)[\tilde v] ].
\]
At $v=0$, this gives 
\begin{equation} \label{pav.f.om.0}
\pa_v f(\om, 0)[\tilde v] 
= \Pi_{Z_N} \pa_u \mF(\om, 0) [ \tilde v + \pa_v w(\om, 0)[\tilde v] ].
\end{equation}
This implies the second identity in \eqref{der.f} because, at $\om = \om_0$, 
$\pa_u \mF(\om_0,0)$ is the operator $L_{\om_0}$,  
whose range is orthogonal to $Z$. 
From \eqref{pav.f.om.0} we also obtain 
\[
\pa_{\om v} f(\om ,0) [\tilde v] 
= \Pi_{Z_N} \pa_{\om u} \mF(\om, 0) [ \tilde v + \pa_v w(\om, 0)[\tilde v] ]
+ \Pi_{Z_N} \pa_u \mF(\om, 0) [ \pa_{\om v} w(\om, 0)[\tilde v] ],
\]
and, at $\om = \om_0$, 
\[
\pa_{\om v} f(\om_0 ,0) [\tilde v] 
= \Pi_{Z_N} \pa_{\om u} \mF(\om_0 , 0) [\tilde v] 
\]
because $\pa_v w(\om_0, 0) = 0$, see \eqref{w.pav.w},
and because the range of $\pa_u \mF(\om_0, 0) = L_{\om_0}$ is orthogonal to $Z$.   
By the definition \eqref{def.mF} of $\mF$, 
one has $\pa_{\om u} \mF(\om,u) = - d (\grad \mI)(u)$, 
and, as already observed in the last identity of \eqref{der.g}, 
$d(\grad \mI)(0) = J_0 \mM$. This proves the third identity in \eqref{der.f}. 

The identities for $F$ and its derivatives follow from \eqref{der.f}, \eqref{der.g} 
and the product rule; for the last identity, we also use the basic fact that 
$\la \Pi_{Z_N} a, a \ra = \la \Pi_{Z_N} a, \Pi_{Z_N} a \ra = | \Pi_{Z_N} a |^2$.
\end{proof}

\begin{lemma} \label{technical.computations} 
Let $v \in V$, with 
$v = \sum_{ (\ell,m)\in S } \hat{v}_{\ell,m} \, \mathtt v_{\ell,m}$ 
for some coefficients $\hat v_{\ell,m} \in \R$, 
where $\mathtt{v}_{\ell,m}$ is defined in \eqref{def.ttv}. Then 
\begin{align}
J_0 \mM v 
& = \sum_{(\ell,m) \in S_N } 
\hat{v}_{\ell,m} m (\ell^2 + \om_0^2 m^2)^{-\frac12} 
\begin{pmatrix}
\omega_0 m \ph_{\ell,m} \\ 
- \ell \ph_{\ell,-m}
\end{pmatrix},
\label{formula.J0.mM.v} 
\\ 
\Pi_{Z_N} J_0 \mM v  
& = \sum_{(\ell,m) \in S_N }
2\om_0 m^2\ell (\ell^2 + \om_0^2 m^2)^{-1} \hat{v}_{\ell,m} 
\, \mathtt v_{\ell,m}, 
\notag 
\\
|\Pi_{Z_N} J_0 \mM v|^2
& = \sum_{(\ell,m) \in S_N }
4 \om_0^2 m^4 \ell^2 (\ell^2 + \om_0^2 m^2)^{-2} \hat{v}_{\ell,m}^2,
\label{approx.I.Z.norm}
\end{align}
where $S_N$ is defined in \eqref{def.SD.SN}.
\end{lemma}

\begin{proof} 
Use the third identity in \eqref{eigenv.rule} and basic calculations.  
\end{proof}

\begin{lemma}[Choice of $\omega(v)$]
\label{lemma:choice.of.omega} 
There exist $\e_1 \in (0, \e_0]$, $b_1, C > 0$ 
and a function $\om : B_{V_N}(\e_1) \to \R$, $v \mapsto \om(v)$, 
where $B_{V_N}(\e_1) := \{ v \in V_N : |v| < \e_1 \}$, 
which is Lipschitz continuous in $B_{V_N}(\e_1)$, 
analytic in $B_{V_N}(\e_1) \setminus \{ 0 \}$, 
such that $\om(0) = \om_0$,
\begin{equation}  \label{IFT.om}
F(\om(v), v) = 0
\end{equation}
for all $v \in B_{V_N}(\e_1)$, 
and, if $(\om,v)$ satisfies $F(\om,v) = 0$ 
with $|\om - \om_0| < b_1$, $v \in B_{V_N}(\e_1)$, 
then $\om = \om(v)$. 
Moreover, the graph $\{ (\om(v) , v) : v \in B_{V_N}(\e_1) \} \subset \R \times V_N$ 
is contained in the open set $\mU_{\e_0} \subset \R \times V$ 
where the function $w$ constructed in Lemma \ref{lemma:range.eq} is defined, and 
\begin{equation} \label{om.estimate}
|\om(v)-\om_0| \leq C |v|
\quad \forall v \in B_{V_N}(\e_1).
\end{equation}
Also, for all $\th \in \T$, one has 
\begin{equation} \label{om.group.action}
\om(\mT_\th v) = \om(v).
\end{equation}
\end{lemma}

\begin{proof} 
By \eqref{der.F}, the expansion of the analytic function $v \mapsto F(\om,v)$ is 
\begin{equation} \label{F.Taylor.in.v}
F(\om,v) = \frac12 \pa_{vv} F(\om, 0)[v, v] + O(v^3).
\end{equation} 
Consider $v \in V_N$, that is, $v \in V$ with coefficients $\hat v_{0,0} = \hat v_{1,0} = 0$, 
and introduce polar coordinates 
$\rho = |v|$, $y = v / |v|$ on $V_N \setminus \{ 0 \}$. 
The function  
\[
\Phi(\om, \rho, y) := \rho^{-2} F(\om, \rho y)
\]
is defined for $|\om - \om_0| < \e_0$, 
$\rho \in (0, \e_0)$, $y$ in the unit sphere $\{ |y| = 1 \}$ of $V_N$.
In fact, replacing $v$ with $\rho y$ in the converging Taylor series of $F(\om,v)$ 
centered at $(\om_0, 0)$, 
one obtains that $\Phi(\om, \rho, y)$ is a well-defined, converging power series 
in the open set $\mD := \{ (\om, \rho, y) : 
|\om - \om_0| < \e_0$, $\rho \in (0, \frac12 \e_0)$, $|y| < 2\}$. 
By \eqref{F.Taylor.in.v}, $\Phi(\om, \rho, y)$ converges to $\frac12 \pa_{vv} F(\om,0)[y,y]$ 
as $\rho \to 0$, and therefore $\Phi$ has a removable singularity at $\rho = 0$. 
Hence $\Phi$ has an analytic extension to the open set 
$\mD_1 := \{ (\om, \rho, y) : |\om - \om_0| < \e_0$, $|\rho| < \frac12 \e_0$, $|y| < 2\}$. 
We also denote this extension by $\Phi$. 
By \eqref{der.F} and \eqref{approx.I.Z.norm}, one has 
\begin{align*}
\Phi(\om_0, 0, y_0) 
& = \tfrac12 \pa_{vv} F(\om_0, 0)[y_0, y_0] = 0, 
\\ 
\pa_\om \Phi (\om_0, 0, y_0) 
& = \tfrac12 \pa_{\om vv} F(\om_0, 0)[y_0, y_0] 
= - \tfrac12 |\Pi_{Z_N} J_0 \mM y_0|^2
\neq 0
\end{align*}
for any $|y_0| = 1$. 
Hence, by the implicit function theorem for analytic functions, 
there exists a function $\Om(\rho, y)$ such that $\Om(0, y_0) = \om_0$ and 
\[
\Phi( \Om(\rho, y), \rho, y) = 0
\]
for all $(\rho,y)$ with $|\rho| < \e_1$, $|y - y_0| < \e_1$, 
for some $\e_1 \in (0, \e_0]$. Moreover, $\e_1$ is independent of $y_0$  
because the unit sphere of $V_N$ is compact. 
Finally, given $v \in V_N$, $|v| < \e_1$, we define 
$\om(v) := \Om(\rho, y)$ with $\rho = |v|$ and $y = v/|v|$ for $v \neq 0$,
and $\om(0) := \om_0$. 
The Lipschitz estimate \eqref{om.estimate} holds because 
$|\om(v) - \om_0| = |\Om(\rho,y) - \Om(0,y)| \leq C \rho$. 

By \eqref{def.f.g}, 
the second identity in \eqref{VD.VN.group.action}, 
the last one in \eqref{conj.mF},
and \eqref{w.group.action},
one has 
\[
\mT_\th f(\om, v) = f(\om, \mT_\th v), \quad \ 
\mT_\th g(\om, v) = g(\om, \mT_\th v).
\]
Hence, by \eqref{def.F}, one has 
\begin{equation} \label{F.group.action}
F(\om, \mT_\th v) = F(\om, v)
\end{equation}
because $\mT_\th^{\,T} \mT_\th = \mathrm{id}$.
By \eqref{IFT.om} with $\mT_\th v$ in place of $v$, one has 
$F(\om(\mT_\th v), \mT_\th v) = 0$. On the other hand, 
by \eqref{F.group.action}, one also has that 
$F(\om(\mT_\th v), \mT_\th v) = F(\om(\mT_\th v), v)$. 
Therefore 
\[
F(\om(\mT_\th v), v) = 0,
\]
and, by the uniqueness property of the implicit function, 
we obtain \eqref{om.group.action}. 
\end{proof}

\begin{remark}
In general, even if the function $\Om(\rho,y)$ in the proof of Lemma \ref{lemma:choice.of.omega}
is analytic around $\rho = 0$,  
the function $\om(v)$ is not analytic around $v=0$, 
because of the singularity of the polar coordinates at the origin;
in fact, $\om(v)$ is not even differentiable at $v=0$, but it is merely Lipschitz continuous. 
For example, consider the variable $v = x = (x_1, \ldots, x_n) \in \R^n$ 
and the function
\[
F(\om, v) = (\om - \om_0) (x_1^2 + \ldots + x_n^2) + x_1^3.
\]
The function $F$ is analytic in $\R^{n+1}$ (it is a polynomial), 
and $F(\om, 0) = 0$, $\pa_v F(\om, 0) = 0$, $\pa_{vv} F(\om_0 , 0) = 0$, 
and $\pa_{\om v v} F(\om_0, 0)[v, v] = 2 |v|^2$, 
so that the identities in \eqref{der.F} are all essentially satisfied.
Solving the equation $F=0$ in the unknown $\om$, we directly calculate the implicit function 
\[
\om(v) = \om_0 - \frac{ x_1^3 }{ x_1^2 + \ldots + x_n^2 },
\]
which is analytic on $\R^n \setminus \{ 0 \}$, 
Lipschitz continuous in $\R^n$, but not differentiable at $v=0$.
The map $\Phi(\om, \rho, y) = \rho^{-2} F(\om, \rho y)$, 
or, more precisely, its analytic extension, is 
\[
\Phi(\om, \rho, y) = (\om - \om_0) (y_1^2 + \ldots + y_n^2) + y_1^3 \rho,
\]
which is analytic in $(\om, \rho, y) \in \R^{n+2}$ ($\Phi$ is a polynomial), 
and the function $\Om(\rho, y)$ is 
\[
\Om(\rho, y) = \om_0 - \frac{ y_1^3 \rho }{ y_1^2 + \ldots + y_n^2 },
\]
which is analytic in a neighborhood of $(\rho,y) = (0, y_0)$, for any $y_0 \in \R^n$, $|y_0| = 1$. 

This basic example shows that $\om(v)$ can be not differentiable at $v=0$ 
even when $F(\om,v)$ is analytic around $(\om_0, 0)$, 
\eqref{der.F} holds, 
and $\Om(\rho, y)$ is analytic around $(0, y_0)$ for all $|y_0| = 1$. 
\end{remark}

\section{The constraint}\label{sec:constraint}

We define the \emph{reduced Hamiltonian} and the \emph{reduced angular momentum} 
\begin{equation} \label{reduced.Hamiltonian} 
\mathcal{H}_{\sigma_0, V_N} (v) := \mathcal{H}_{\sigma_0}( u(v) ), \quad \ 
\mI_{V_N} (v) := \mI( u(v) ), \quad \ 
u(v) := v + w( \om(v), v),
\end{equation}
and, for $a \in \R$, we define the set 
\begin{equation} \label{reduced.constraint}
\mS_{V_N} (a) := \{ v \in B_{V_N}(\e_1) : \ \mI_{V_N} (v)=a \},
\end{equation}
where $\om(v)$ is the function constructed in Lemma \ref{lemma:choice.of.omega},
defined for $v \in B_{V_N}(\e_1) \subset V_N$. 
Since $\om(v)$ is Lipschitz around $v=0$, the function $v \mapsto w(\om(v), v)$ is also Lipschitz; 
in fact, we observe in the next lemma that its regularity is higher. 

\begin{lemma} \label{lemma:reg.u(v)}
The function $B_{V_N}(\e_1) \to W^s$, $v \mapsto w(\om(v), v)$ 
is analytic in $B_{V_N}(\e_1) \setminus \{ 0 \}$,
it is of class $C^1 ( B_{V_N}(\e_1) )$, 
and its differential is Lipschitz continuous in $B_{V_N}(\e_1)$. 
Moreover $w(\om(\mT_\th v), \mT_\th v) = \mT_\th w(\om(v), v)$ for all $\th \in \T$.
\end{lemma}

\begin{proof}
The function $v \mapsto w(\om(v),v)$ is Lipschitz continuous in $B_{V_N}(\e_1)$
and analytic in $B_{V_N}(\e_1) \setminus \{ 0 \}$ 
by Lemmas \ref{lemma:range.eq}, \ref{lemma:choice.of.omega}.
By \eqref{w.estimate} and \eqref{om.estimate}, one has 
\begin{equation} \label{w.comp.estimate}
\| w(\om(v), v) \|_{W^s} \leq C |v|^2
\end{equation}
for all $v \in B_{V_N}(\e_1)$. 
Hence the function $v \mapsto w(\om(v),v)$ is differentiable at $v=0$ 
with zero differential. 
Its differential at any point $v \in B_{V_N}(\e_1) \setminus \{ 0 \}$ 
in direction $\tilde v \in V_N$ is 
\begin{equation}  \label{der.w.comp}
\pa_v \{ w( \om(v) , v) \} [\tilde v] 
= (\pa_\om w)( \om(v) , v) \om'(v)[\tilde v] 
+ (\pa_v w)( \om(v), v)[\tilde v].
\end{equation}
For $v \to 0$, one has 
$(\om(v), v) \to (\om_0,0)$ because the function $\om(v)$ in Lemma \ref{lemma:choice.of.omega}
is continuous,
$(\pa_v w)(\om(v), v) \to \pa_v w(\om_0, 0) = 0$ 
and $(\pa_\om w)(\om(v), v) \to \pa_\om w(\om_0, 0) = 0$  
because the function $w$ in Lemma \ref{lemma:range.eq} is analytic 
and by \eqref{w.pav.w}, 
while $\om'(v)[\tilde v]$, which is defined for $v \neq 0$, remains bounded as $v \to 0$
because $\om(v)$ is Lipschitz. 
Hence $\pa_v \{ w( \om(v) , v) \} \to 0$, which implies that $w(\om(v), v)$ 
is of class $C^1 ( B_{V_N}(\e_1) )$. 
Moreover, 
$|\om'(v)[\tilde v]| \leq C |\tilde v|$ for $0 < |v| < \e_1$ 
because $\om(v)$ is Lipschitz, 
and 
\[
\| (\pa_\om w)( \om(v) , v) \|_{W^s} \leq C |v|, \quad \  
\| (\pa_v w)( \om(v), v)[\tilde v] \|_{W^s} \leq C |v| |\tilde v|
\]
because the function $w$ in Lemma \ref{lemma:range.eq} is analytic 
and by \eqref{om.estimate}.
Hence, by \eqref{der.w.comp}, 
\[
\| \pa_v \{ w( \om(v) , v) \} [\tilde v] \|_{W^s}
\leq C |v| |\tilde v|
\]
for $v \in B_{V_N}(\e_1) \setminus \{ 0 \}$, 
so that the map $v \mapsto \pa_v \{ w( \om(v) , v) \}$ is Lipschitz continuous around $v=0$. 

The last property in the statement follows from \eqref{w.group.action} and \eqref{om.group.action}.
\end{proof}

\begin{lemma} \label{lemma:reg.reduced.mH.mI}
The functionals $\mH_{\sigma_0,V_N} : B_{V_N}(\e_1) \to \R$ 
and $\mI_{V_N} : B_{V_N}(\e_1) \to \R$ 
are analytic in $B_{V_N}(\e_1) \setminus \{ 0 \}$
and of class $C^2(B_{V_N}(\e_1))$, 
and their second order differentials are Lipschitz continuous in $B_{V_N}(\e_1)$.
Moreover $\mH_{\sigma_0,V_N} \circ \mT_\th = \mH_{\sigma_0,V_N}$ 
and $\mI_{V_N} \circ \mT_\th = \mI_{V_N}$ for all $\th \in \T$. 
\end{lemma}

\begin{proof}
From Lemmas 
\ref{lemma:mF.Hs}, \ref{lemma:reg.u(v)} and \ref{lemma:group.action}
one deduces the analyticity, the $C^1$ regularity with Lipschitz differentials 
and the invariance with respect to the group action $\mT_\th$. 
To prove the higher regularity, we proceed like in the proof of Lemma \ref{lemma:reg.u(v)}.
By \eqref{reduced.Hamiltonian}, 
the differential of $\mI_{V_N}$ at a point $v \in B_{V_N}(\e_1)$ in direction $\tilde v \in V_N$ is 
\[
\mI_{V_N}'(v)[\tilde v] 
= \mI'(u(v))[ u'(v)[ \tilde v] ], 
\quad \ 
u'(v)[\tilde v] = \tilde v + \pa_v \{ w(\om(v),v) \}[\tilde v],
\]
and the map $v \mapsto \mI_{V_N}'(v)$ is Lipschitz. 
At $v \neq 0$, its differential in direction $\tilde z \in V_N$ is 
\[
\mI_{V_N}''(v)[\tilde z, \tilde v] 
= \mI''(u(v))[ u'(v)[ \tilde z], u'(v)[ \tilde v] ] 
+ \mI'(u(v))[ u''(v)[ \tilde z, \tilde v] ].
\]
As $v \to 0$, one has $u(v) \to 0$, $u'(v)[\tilde v] \to \tilde v$, 
$\mI'(u(v)) \to \mI'(0) = 0$, and $\mI''(u(v)) \to \mI''(0)$, 
while the bilinear map $u''(v) = \pa_{vv} \{ w( \om(v), v ) \}$ 
remains bounded, i.e., $\| u''(v)[ \tilde z, \tilde v] \|_{W^s} \leq C |\tilde z| |\tilde v|$
uniformly as $v \to 0$, because the map $v \mapsto u'(v)$ is Lipschitz. 
Hence $\mI_{V_N}''(v)[\tilde z, \tilde v]$ converges to 
$\mI''(0)[\tilde z, \tilde v]$ as $v \to 0$. 
Similarly, one proves that 
\[
| \mI_{V_N}'(v)[\tilde v] - \mI''(0)[v, \tilde v] | 
\leq C |v|^2 |\tilde v|
\]
for $v \neq 0$. This implies that the map $v \mapsto \mI_{V_N}'(v)$ is differentiable at $v=0$, 
with differential $\mI_{V_N}''(0) = \mI''(0)$. 
Also, from the limit already proved it follows that $\mI_{V_N}$ is of class $C^2$. 
The Lipschitz estimate for the second order differential
\[
| \mI_{V_N}''(v)[ \tilde z, \tilde v] - \mI_{V_N}''(0)[ \tilde z, \tilde v] | 
\leq C |v| |\tilde z| |\tilde v|
\]
is proved similarly. The proof for $\mH_{\sigma_0, V_N}$ is analogous.
\end{proof}

We recall that $V$ is a finite-dimensional space, 
which can be identified with the Euclidean space $\R^{2n+2}$, 
where $2n+2$ is the cardinality of the set $S$ in \eqref{def.S}, 
and $V_N = \{ v \in V : \hat v_{0,0} = \hat v_{1,0} = 0 \}$ 
is a linear subspace of $V$ of dimension $2n$. 
Thus, $n$ is the number of elements $(\ell, m)$ of $S$ with $\ell \geq 2$ and $m \geq 1$.
The set $\mS_{V_N}(a) \subset V_N$ will play the role of a \emph{constraint} for our variational problem. 
To study its geometrical and topological properties, 
we study the functional $\mI_{V_N}$ in more detail. 

\begin{lemma}\label{IZ.formula} 
The functional $\mI_{V_N} : B_{V_N}(\e_1) \to \R$ defined in \eqref{reduced.Hamiltonian} 
satisfies
\begin{equation}\label{IZ.approx}
\mathcal{I}_{V_N} (v) = \mathcal{I}_0 (v) + \mR(v), 
\quad \ |\mR(v)| \leq C |v|^3,
\end{equation}
where 
\begin{equation} \label{def.mI.0}
\mathcal{I}_0(v) := 
\frac12 \la J_0 \mM v , v \ra_{L^2(\S^2)} 
= \sum_{(\ell,m) \in S_N} \frac{ \omega_0 m^2 \ell }{ \ell^2 + \om_0^2 m^2 } \, \hat{v}_{\ell,m}^2,
\end{equation}
for $v = \sum_{ (\ell,m) \in S_N } \hat v_{\ell,m} \, \mathtt v_{\ell,m}$, 
where $J_0$ is defined in \eqref{mF.mF.old} 
and $S_N$ in \eqref{def.SD.SN}.
\end{lemma}

\begin{proof} 
The Taylor expansion around $u=0$ 
of the analytic functional $\mI$ defined in \eqref{def.mI} is 
\begin{equation} \label{Taylor.mI}
\mI(u) = \mI_0(u) + O(u^3),
\end{equation} 
with $\mI_0$ defined in \eqref{def.mI.0}, 
because $\mI(0) = 0$, $\grad \mI(0) = 0$, and 
\[ 
\frac12 \mI''(0)[u,u] 
= \frac12 \la d(\grad \mI)(0)[u], u \ra_{L^2(\S^2)} 
= \frac12 \la J_0 \mM u , u \ra_{L^2(\S^2)} 
= \mI_0(u).
\]
Note that the identity $d(\grad \mI)(0) = J_0 \mM$ was already observed in \eqref{der.g}. 
Alternatively, \eqref{Taylor.mI} can also be deduced from \eqref{def.mI} by observing that
\[
\mI(u) = \int_{\S^2} (1 + 2 \eta + \eta^2) (\mM \eta) \beta \, d\sigma 
= \la \mM \eta , \beta \ra_{L^2(\S^2)} + O(u^3), 
\quad u = (\eta, \beta),
\]
and that $\la \mM \eta , \beta \ra_{L^2(\S^2)} = \frac12 \la J_0 \mM u , u \ra_{L^2(\S^2)}$
by \eqref{mM.is.anti.symm}. 
Recalling \eqref{w.comp.estimate}, 
plugging $u(v) = v + w(\om(v),v) = v + O(|v|^2)$ 
into the expansion \eqref{Taylor.mI} gives the first identity in \eqref{IZ.approx}. 
The last identity in \eqref{IZ.approx} follows from 
the formula of $J_0 \mM v$ in \eqref{formula.J0.mM.v}
and the definition \eqref{def.ttv} of $\mathtt v_{\ell,m}$. 
\end{proof}

From expansion \eqref{IZ.approx}, by the implicit function theorem, 
we obtain that the set $\mS_{V_N}(a)$ is diffeomorphic to a sphere of $V_N$, 
with a diffeomorphism preserving the $\mT_\th$ group action.  
More precisely, we prove the following lemma.

\begin{lemma} \label{lemma:geom} 
There exist a constant $\e_2 \in (0, \e_1]$ 
and a map $\psi : B_{V_N}(\e_2) \to B_{V_N}(\e_1)$ such that 
\begin{equation} \label{mI.VN.psi}
\mI_{V_N}(\psi(v)) = |v|^2 
\end{equation}
for all $v \in B_{V_N}(\e_2)$. 
The map $\psi$ is analytic in $B_{V_N}(\e_2) \setminus \{ 0 \}$, 
it is of class $C^1( B_{V_N}(\e_2) )$, with Lipschitz differential,  
it is a diffeomorphism of $B_{V_N}(\e_2)$ onto its image $\psi( B_{V_N}(\e_2) )$, 
which is an open neighborhood of $\psi(0)=0$ contained in $B_{V_N}(\e_1)$,   
and 
\begin{equation} \label{psi.group.action}
\psi \circ \mT_\th = \mT_\th \circ \psi
\end{equation}
for all $\th \in \T$. 
Moreover, there exists $a_0 > 0$ such that, for all $a \in (0, a_0)$, 
the set $\mS_{V_N}(a)$ in \eqref{reduced.constraint} is the diffeomorphic image 
\[
\mS_{V_N}(a) = \psi \big( \mathtt{S}(a) \big) 
\]
of the sphere 
\begin{equation} \label{def.mS.0.VN.a}
\mathtt{S}(a) := \{ v \in V_N : |v|^2 = a \},
\end{equation}
and thus $\mS_{V_N}(a)$ is an analytic connected compact manifold of dimension $2n-1$ embedded in $V_N$.
Its tangent and normal space at a point $v \in \mS_{V_N}(a)$ are
\begin{align} 
T_v( \mS_{V_N}(a) ) 
& = \{ \tilde v \in V_N : \la \grad \mI_{V_N} (v) , \tilde v \ra_{L^2(\S^2)} = 0 \}, 
\label{tangent.space}
\\
N_v( \mS_{V_N}(a) ) 
& = \{ \lm \grad \mI_{V_N} (v) : \lm \in \R \},
\label{normal.space}
\end{align}
where $\grad \mI_{V_N}(v) \neq 0$ for all $v \in \mS_{V_N}(a)$, all $a \in (0, a_0)$.
\end{lemma}

\begin{proof}
Define the diagonal linear map 
\begin{equation} \label{def.Lm}
\Lm : V_N \to V_N, \quad 
\Lm \mathtt{v}_{\ell,m} := \lm_{\ell,m} \mathtt{v}_{\ell,m}, \quad 
\lm_{\ell,m} := (\ell^2 + \om_0^2 m^2)^{\frac12} (\om_0 m^2 \ell)^{-\frac12}.
\end{equation}
Let $|\Lm| := \max \lm_{\ell,m}$, 
so that $|\Lm v| \leq |\Lm| |v|$ for all $v \in V_N$.  
By \eqref{norm.ttv} and \eqref{IZ.approx}, one has 
\begin{equation} \label{hat.square}
\mI_0(\Lm v) = \sum_{(\ell,m) \in S_N} \hat v_{\ell,m}^2 = |v|^2.
\end{equation}
Given $v$, we look for a real number $\mu$ such that 
\begin{equation} \label{19:22}
\mI_{V_N}( (1+\mu) \Lm v ) = |v|^2. 
\end{equation}
We look for $\mu$ in the interval $[-\delta, \delta]$, 
with $\delta = \frac14$, 
and we assume that $|v| < \e_2$, where $\e_2 \in (0, \e_1]$ is such that 
$\frac54 |\Lm| \e_2 < \e_1$, so that, for $v \in B_{V_N}(\e_2)$, 
the point $(1+\mu) \Lm v$ is in the ball $B_{V_N}(\e_1)$ where $\mI_{V_N}$ is defined. 
Let $\mR$ be the remainder in \eqref{IZ.approx}. 
By \eqref{hat.square}, one has 
$\mI_0( (1+\mu) \Lm v ) = (1+\mu)^2 |v|^2$,  
and, since $\mI_{V_N}= \mI_0 + \mR$, 
\eqref{19:22} becomes 
\begin{equation} \label{19:23}
(2 \mu + \mu^2) |v|^2 + \mR( (1+\mu) \Lm v ) = 0.
\end{equation}
For $v \neq 0$, \eqref{19:23} is the fixed point equation $\mu = \mK(\mu)$
for the unknown $\mu$, where 
\begin{equation} \label{def.mK}
\mK(\mu) := - \frac{\mu^2}{2} - \frac{\mR( (1+\mu) \Lm v)}{2 |v|^2}.
\end{equation}
By the inequality in \eqref{IZ.approx}, 
for some constants $C_1, C_2$ one has
\[
| \mK(\mu) | \leq \frac{\delta^2}{2} + C_1 \e_2 \leq \delta, 
\quad \ 
| \mK'(\mu) | \leq \delta + C_2 \e_2 \leq \frac12 
\]
for all $|\mu| \leq \delta$, $|v| < \e_2$,  
provided $\e_2$ is sufficiently small, namely
$C_1 \e_2 \leq \frac12 \delta$ and $C_2 \e_2 \leq \frac14$. 
Hence, by the contraction mapping theorem, 
in the interval $[- \delta, \delta]$ there exists a unique fixed point of $\mK$, 
which we denote by $\mu(v)$. 
Hence 
\begin{equation}  \label{def.psi.in.the.proof}
\mI_{V_N}( \psi (v) ) = |v|^2, \quad \ 
\psi(v) := (1 + \mu(v)) \Lm v,  
\end{equation}
for all $v \in B_{V_N}(\e_2) \setminus \{ 0 \}$. 
From the implicit function theorem 
applied to equation \eqref{19:23} around any pair $(v, \mu(v))$ 
it follows that the map $v \mapsto \mu(v)$ is analytic in $B_{V_N}(\e_2) \setminus \{ 0 \}$. 
Moreover, 
$|\mu(v)| = |\mK(\mu(v))| \leq \frac12 \mu^2(v) + C_1 |v| \leq \frac18 |\mu(v)| + C_1 |v|$, 
whence 
\[
|\mu(v)| \leq C |v|.
\]
Thus, defining $\mu(0) := 0$, the function $\mu(v)$ is also Lipschitz in $B_{V_N}(\e_2)$. 
As a consequence, the function $\psi$ is analytic in $B_{V_N}(\e_2) \setminus \{ 0 \}$
and Lipschitz in $B_{V_N}(\e_2)$. In addition, $|\psi(v) - \psi(0) - \Lm v| \leq C |v|^2$, 
which means that $\psi$ is differentiable also at $v=0$, 
with differential $\psi'(0)[\tilde v] = \Lm \tilde v$. 
At $v \neq 0$, the differential is 
\[
\psi'(v)[\tilde v] = \mu'(v)[\tilde v] \Lm v + (1 + \mu(v)) \Lm \tilde v,
\]
and $\psi'(v) \to \psi'(0) = \Lm$ as $v \to 0$ because $\mu(v) \to 0$, $\Lm v \to 0$, 
and $|\mu'(v)[\tilde v]| \leq C |\tilde v|$ uniformly as $v \to 0$.
Thus $\psi$ is of class $C^1$ in $B_{V_N}(\e_2)$. 
Moreover 
\[
|\psi'(v)[\tilde v] - \psi'(0)[\tilde v]| \leq C |v| |\tilde v|,
\]
i.e., the differential map $v \mapsto \psi'(v)$ is Lipschitz continuous. 

The function $\psi$ is a diffeomorphism of open sets of $V_N$, and, 
for each $a \in (0, a_0)$ sufficiently small, one has 
\begin{align*}
\mS_{V_N}(a) 
& = \{ v \in B_{V_N}(\e_1) : \mI_{V_N}(v) = a \} 
\\ 
& = \{ v = \psi(y) : y \in B_{V_N}(\e_2), \ a = \mI_{V_N}(v) = \mI_{V_N}(\psi(y)) = |y|^2 \}
\\ 
& = \psi ( \{ y \in V_N : |y|^2 = a \} ),
\end{align*}
namely $\mS_{V_N}(a)$ is the image of the sphere $\{ |y|^2 = a \}$ by the diffeomorphism $\psi$. 

Since the eigenvalues in \eqref{def.Lm} satisfy 
$\lm_{\ell, -m} = \lm_{\ell,m}$, by \eqref{mT.th.v} we have
\begin{align*}
\Lm (\mT_\th v) 
& = \sum_{(\ell,m) \in S_N} \{ \cos(m \th) \hat v_{\ell,m} 
+ \sin(m \th) \hat v_{\ell,-m} \} \lm_{\ell,m} \mathtt{v}_{\ell,m}
\\ 
& = \sum_{(\ell,m) \in S_N} \{ \cos(m \th) (\lm_{\ell,m} \hat v_{\ell,m})
+ \sin(m \th) (\lm_{\ell,-m} \hat v_{\ell,-m}) \} \mathtt{v}_{\ell,m} 
= \mT_\th (\Lm v)
\end{align*}
for all $v \in V_N$, i.e., 
\begin{equation} \label{Lm.group.action} 
\Lm \mT_\th = \mT_\th \Lm.
\end{equation} 
By \eqref{def.mT.theta} and \eqref{def.norma.V}, 
with the change of integration variable $y = R(\th)x$ on $\S^2$, 
one proves that 
\begin{equation} \label{norm.group.action}
|\mT_\th v|^2 = |v|^2.
\end{equation}
Alternatively, the same identity can also be proved using \eqref{mT.th.v} 
and the fact that $\{ \mathtt{v}_{\ell,m} : (\ell,m) \in S_N \}$ is an orthonormal basis of $V_N$.  
From \eqref{Lm.group.action}, \eqref{hat.square}, and \eqref{norm.group.action}, one has 
\[
\mI_0( \mT_\th \Lm v ) = \mI_0( \Lm \mT_\th v ) 
= |\mT_\th v|^2 = |v|^2 = \mI_0( \Lm v)
\] 
for all $v \in V_N$, 
and this implies that $\mI_0 \circ \mT_\th = \mI_0$ 
because $\{ \Lm v : v \in V_N \} = V_N$. 
By Lemma \ref{lemma:reg.reduced.mH.mI}, the functional $\mI_{V_N}$ has the same invariance property,  
and therefore the difference $\mR = \mI_{V_N} - \mI_0$ also satisfies
\begin{equation}  \label{mR.group.action}
\mR \circ \mT_\th = \mR
\end{equation}
for all $\th \in \T$. 
Now denote by $\mK(\mu, v)$ the scalar quantity $\mK(\mu)$ in \eqref{def.mK}. 
By the contraction mapping theorem, we have proved that $\mK(\mu(v), v) = \mu(v)$ 
for all $v \in B_{V_N}(\e_2) \setminus \{ 0 \}$, 
and, if a pair $(\mu, v) \in [- \delta, \delta] \times (B_{V_N}(\e_2) \setminus \{ 0 \})$ 
satisfies $\mK(\mu, v) = \mu$, then $\mu = \mu(v)$. 
Moreover, by \eqref{norm.group.action} and \eqref{mR.group.action}, one has
$\mK(\mu, \mT_\th v) = \mK(\mu, v)$ for all pairs $(\mu, v)$, all $\th \in \T$. 
Then 
\[
\mu(v) = \mK( \mu(v), v) = \mK( \mu(v), \mT_\th v),
\]
whence $\mu(v) = \mu( \mT_\th v)$. 
This identity, together with  \eqref{Lm.group.action}, 
gives \eqref{psi.group.action}.
\end{proof}

\section{Variational structure of the reduced non-degenerate bifurcation equation}
\label{sec:reduced.variational}

In this section we prove that the projection on the non-degenerate subspace $Z_N$ 
of the bifurcation equation \eqref{Bifurcation.equation} 
restricted to the graph $\om = \om(v)$ has a variational structure, 
namely it is the critical point equation of a functional. 

The reduced Hamiltonian $\mH_{\sigma_0, V_N}$ and the reduced angular momentum $\mI_{V_N}$ 
are defined in \eqref{reduced.Hamiltonian}, 
and the function $\om(v)$ is defined in Lemma \ref{lemma:choice.of.omega}. 
Now for $v \in B_{V_N}(\e_1)$ we define the functional 
\begin{equation} \label{def.mE}
\mE_{V_N}(v) := \mH_{\sigma_0, V_N}(v) - \om(v) \big( \mI_{V_N}(v) - a \big).
\end{equation}

\begin{lemma} For all $a \in \R$, the functional $\mE_{V_N} : B_{V_N}(\e_1) \to \R$ 
is Lipschitz continuous in $B_{V_N}(\e_1)$ and analytic in $B_{V_N}(\e_1) \setminus \{ 0 \}$.
Moreover, for all $\th \in \T$, one has  
\begin{equation} \label{mE.VN.group.action}
\mE_{V_N} \circ \mT_\th = \mE_{V_N}.
\end{equation}
\end{lemma}

\begin{proof}
It follows from Lemmas \ref{lemma:choice.of.omega} and \ref{lemma:reg.reduced.mH.mI}.
\end{proof}

\begin{lemma} 
The differential of the functional $\mE_{V_N}$ 
at a point $v \in B_{V_N}(\e_1) \setminus \{ 0 \}$ 
in any direction $\tilde v \in V_N$ is 
\begin{equation} \label{der.mE.V}
\mE_{V_N}'(v)[\tilde v] 
= \la \grad \mE_{V_N}(v) , \tilde v \ra_{L^2(\S^2)}  
= \la \Pi_{Z_N} \mF (\om(v), u(v)) , \tilde v \ra_{L^2(\S^2)}  
- \omega'(v)[\tilde v] \big( \mI_{V_N} (v) - a \big).
\end{equation}
\end{lemma}  

\begin{proof} 
We calculate 
\begin{align*}
& \mE_{V_N}'(v)[\tilde v] 
= \mathcal{H}_{\sigma_0, V_N}'(v)[\tilde v] - \omega'(v)[\tilde v] \big( \mI_{V_N} (v) - a \big) 
- \om(v) \mI_{V_N}'(v) [\tilde v] 
\notag \\
& \quad 
= \la (\grad \mH_{\sigma_0})(u(v)) , u'(v)[\tilde v] \ra_{L^2(\S^2)} 
- \omega'(v)[\tilde v] \big( \mI_{V_N}(v) - a \big)  
- \om(v) \la (\grad \mI)(u(v)) , u'(v)[\tilde v] \ra_{L^2(\S^2)} 
\notag \\
& \quad 
= \la \mF (\om(v), u(v)) , u'(v)[\tilde v] \ra_{L^2(\S^2)}  
- \omega'(v)[\tilde v] \big( \mI_{V_N}(v) - a \big)  
\notag \\
& \quad 
= \la \mF (\om(v), u(v)) , \tilde v \ra_{L^2(\S^2)}  
+ \la \mF (\om(v), u(v)) , \pa_v \{ w(\om(v), v) \} [\tilde v] \ra_{L^2(\S^2)}  
- \omega'(v)[\tilde v] \big( \mI_{V_N} (v) - a \big)  
\notag \\
& \quad 
= \la \Pi_{Z_N} \mF (\om(v), u(v)) , \tilde v \ra_{L^2(\S^2)}  
+ \la \Pi_R \mF (\om(v), u(v)) , \pa_v \{ w(\om(v), v) \} [\tilde v] \ra_{L^2(\S^2)}  
\notag \\ 
& \quad \quad \ 
- \omega'(v)[\tilde v] \big( \mI_{V_N} (v) - a \big)  
\notag \\
& \quad 
= \la \Pi_{Z_N} \mF (\om(v), u(v)) , \tilde v \ra_{L^2(\S^2)}  
- \omega'(v)[\tilde v] \big( \mI_{V_N} (v) - a \big),  
\end{align*}
where we have used the definition \eqref{reduced.Hamiltonian} of $\mH_{\sigma_0, V_N}, \mI_{V_N}$, $u(v)$,
the chain rule, 
the definition \eqref{def.mF} of $\mF$, 
the fact that $\tilde v \in V_N$ and $\pa_v \{ w(\om(v), v) \}[\tilde v] \in W$,
the identities in \eqref{Z.is.V.R.is.W}, 
and the range equation \eqref{IFT.w}.
\end{proof}

The manifold $\mS_{V_N}(a)$ has been defined in \eqref{reduced.constraint} 
and studied in Lemma \ref{lemma:geom} for $a \in (0,a_0)$.
For points $v$ on $\mS_{V_N}(a)$ we immediately deduce the following lemma. 

\begin{lemma} \label{lemma:variational.nature.ZN}
Let $a \in (0, a_0)$, where $a_0$ is given in Lemma \ref{lemma:geom}. 
For any $v \in \mS_{V_N}(a)$, one has 
\begin{equation} \label{grad.mE.V}
\grad \mE_{V_N}(v) = \Pi_{Z_N} \mF (\om(v), u(v)).
\end{equation}
\end{lemma}  

\begin{proof}
The term $(\mI_{V_N}(v) - a)$ in \eqref{der.mE.V} vanishes 
because $v \in \mS_{V_N}(a)$.
\end{proof}

\begin{lemma}[Constrained critical points solve the reduced bifurcation equation on $Z_N$]
\label{lemma:sol.bif.ZN}
Let $a \in (0, a_0)$, where $a_0$ is given in Lemma \ref{lemma:geom}. 
Suppose that $v \in \mS_{V_N}(a)$ is a constrained critical point 
of the functional $\mE_{V_N}$ on the constraint $\mS_{V_N}(a)$. 
Then $v$ is a critical point of $\mE_{V_N}$ in $V_N$, and 
\begin{equation} \label{bif.eq.ZN}
\Pi_{Z_N} \mF(\om(v), u(v)) = 0.
\end{equation}
\end{lemma}

\begin{proof}
Let $v$ be a constrained critical point of $\mE_{V_N}$ on the constraint $\mS_{V_N}(a)$. 
Then 
\[
\mE_{V_N}'(v)[\tilde v] 
= \la \grad \mE_{V_N}(v) , \tilde v \ra_{L^2(\S^2)} 
= 0 
\]
for all tangent vectors $\tilde v \in T_v ( \mS_{V_N}(a) )$, 
and therefore $\grad \mE_{V_N}(v)$ (which is a vector of $V_N$)
belongs to the normal space $N_v ( \mS_{V_N}(a) )$, that is, by \eqref{normal.space},
\begin{equation} \label{Lag.mult.mE.VN}
\grad \mE_{V_N}(v) = \lm \grad \mI_{V_N}(v)
\end{equation}
for some Lagrange multiplier $\lm \in \R$.
Then, using \eqref{IFT.om}, \eqref{def.F}, \eqref{def.f.g}, 
the definition of $u(v)$ in \eqref{reduced.Hamiltonian},
\eqref{grad.mE.V}, and \eqref{Lag.mult.mE.VN}, 
we obtain
\begin{align} \label{0=lm.scal.prod}
0 & = F(\om(v), v) 
= \la \Pi_{Z_N} \mF (\om(v), u(v)) , (\grad \mI)(u(v)) \ra_{L^2(\S^2)}
= \lm \la \grad \mI_{V_N}(v) , (\grad \mI)(u(v)) \ra_{L^2(\S^2)}.
\end{align}
By \eqref{IZ.approx}, \eqref{Taylor.mI}, and \eqref{w.comp.estimate}, one has 
\[
\grad \mI_{V_N}(v) = \grad \mI_0(v) + O(|v|^2), 
\quad 
(\grad \mI)(u(v)) = \grad \mI_0(v) + O(|v|^2),
\]
and
\[
\la \grad \mI_{V_N}(v) , (\grad \mI)(u(v)) \ra_{L^2(\S^2)} 
= |\grad \mI_0(v)|^2 + O(|v|^3) 
\geq \tfrac12 |\grad \mI_0(v)|^2 
> 0. 
\]
This means that the coefficient of $\lm$ in \eqref{0=lm.scal.prod} is nonzero, 
whence $\lm = 0$. Thus, by \eqref{Lag.mult.mE.VN}, $\grad \mE_{V_N}(v) = 0$, 
and \eqref{bif.eq.ZN} follows from \eqref{grad.mE.V}.
\end{proof}

Since $\mS_{V_N}(a)$ is a smooth compact manifold, 
the functional $\mE_{V_N}$ constrained to $\mS_{V_N}(a)$ 
has at least one minimum point and one maximum point 
(if $\mE_{V_N}$ is constant, there are infinitely many constrained critical points). 
In fact, in Lemma \ref{lemma:n.critical.orbits} we give a more accurate lower estimate 
on the number of distinct constrained critical points of $\mE_{V_N}$ on $\mS_{V_N}(a)$, 
using the following result from \cite{Mahwin.Willem}.

\begin{lemma}[Lemma 6.10 in \cite{Mahwin.Willem}, Sec.\ 6.4] 
\label{lemma:6.10.Mawhin.Willem}
Let $A$ be an open $\S^1$-invariant neighborhood of $\S^{2k-1}$ and
let $\ph \in C^1(A, \R)$ be a $\S^1$-invariant function. 
Then there exist at least $k$ $\S^1$-orbits of critical points of $\ph$ 
restricted to $\S^{2k-1}$.
\end{lemma}

\begin{lemma} \label{lemma:n.critical.orbits}
For any $a \in (0, a_0)$, the functional $\mE_{V_N}$ has at least $n$ distinct 
orbits $\{ \mT_\th v^{(i)} : \th \in \T \}$, 
$i = 1, \ldots, n$, 
of constrained critical points on the constraint $\mS_{V_N}(a)$.
\end{lemma}

\begin{proof}
Fix $a \in (0, a_0)$. 
By Lemma \ref{lemma:geom}, one has $\mS_{V_N}(a) = \psi (\mathtt S(a) )$. 
Consider the function
\begin{equation}  \label{def.comp.mE.VN.psi}
f : B_{V_N}(\e_2) \to \R, \quad \ f(y) := \mE_{V_N} (\psi(y)). 
\end{equation}
The differential of $f$ at a point $y \in B_{V_N}(\e_2)$  
in direction $\tilde y \in V_N$ is 
\[
f'(y)[\tilde y] = \mE_{V_N}'(\psi(y))[ \psi'(y)[\tilde y]],
\]
and, for $y \in \mathtt S(a)$, 
$\tilde y$ is in the tangent space $T_y( \mathtt S(a) )$ 
iff $\psi'(y)[\tilde y]$ is in the tangent space $T_{\psi(y)} ( \mS_{V_N}(a) )$. 
Hence a point $v = \psi(y) \in \mS_{V_N}(a)$ 
is a constrained critical point of $\mE_{V_N}$ 
on the constraint $\mS_{V_N}(a)$ iff 
$y \in \mathtt S(a)$ is a constrained critical point of $f$ 
on the constraint $\mathtt S(a)$. 
By \eqref{psi.group.action} and \eqref{mE.VN.group.action}, one has 
\begin{equation} \label{f.group.action}
f \circ \mT_\th = f
\end{equation}
for all $\th \in \T$. 
Then, by Lemma \ref{lemma:6.10.Mawhin.Willem}, 
$f$ has at least $n$ orbits of critical points on $\mathtt S(a)$.
\end{proof}

Compared with the basic observation that $\mE_{V_N}$ 
constrained to $\mS_{V_N}(a)$ always has at least one minimum and one maximum point, 
Lemma \ref{lemma:n.critical.orbits} provides more solutions 
even in the case $n=1$, as it gives the existence of an \emph{orbit} of solutions
(therefore infinitely many solutions). 
In fact, for $n=1$ we have the following result.

\begin{lemma} \label{lemma:n=1}
If $n=1$, then, for every $a \in (0, a_0)$, 
the constraint $\mS_{V_N}(a)$ is diffeomorphic to $\S^1 = \T$  
and it consists of one orbit of critical points of $\mE_{V_N}$.  
Moreover the orbit $\mS_{V_N}(a)$ depends analytically on $a$ in the interval $(0, a_0)$. 
\end{lemma}

\begin{proof}
For each value $a \in (0, a_0)$, 
the functional $\mE_{V_N}$ is constant on the constraint $\mS_{V_N}(a)$ by \eqref{mE.VN.group.action}. 
Hence the orbit of critical points of $\mE_{V_N}$ 
is exactly the level set $\mS_{V_N}(a)$ of the restricted angular momentum $\mI_{V_N}$, 
which depends analytically on $a$.
\end{proof}

\section{Solution of the bifurcation equation in the degenerate directions} 
\label{sec:sol.bif.deg}

In this section we deal with the degenerate subspace $V_D = Z_D$, 
and we prove that the constrained critical points of Lemma \ref{lemma:sol.bif.ZN}
give solutions of the bifurcation equation also on $Z_D$.

\begin{lemma}[Solution of the bifurcation equation]
\label{lemma:sol.bif}
Let $a \in (0, a_0)$, and suppose that $v \in \mS_{V_N}(a)$ is a constrained critical point 
of the functional $\mE_{V_N}$ on the constraint $\mS_{V_N}(a)$. 
Then $(\om(v), u(v))$ solve the (full) bifurcation equation
\begin{equation} \label{bif.eq.ZN.bis}
\Pi_Z \mF(\om(v), u(v)) = 0.
\end{equation}
\end{lemma}

\begin{proof} 
By \eqref{bif.eq.ZN} and \eqref{IFT.w}, one has 
\[
\mF(\om(v), u(v)) = \Pi_{Z_D} \mF(\om(v), u(v)).
\]
This means that $\mF(\om(v), u(v))$ belongs to the degenerate space $Z_D$, 
which is the 2-dimensional linear space generated by $\mathtt v_{0,0}$ and $\mathtt v_{1,0}$. 
By the first orthogonality relation in \eqref{eq:orthogonal.vectors}, 
recalling the definition \eqref{def.ttv} of $\mathtt v_{0,0}$, 
where $\ph_{0,0}$ is a constant, one has 
\[
\la \mF(\om(v), u(v)) , \, \mathtt v_{0,0} \ra_{L^2(\S^2)} = 0.
\]
Hence 
\begin{equation}  \label{mF.mu.ttv.1.0}
\mF(\om(v), u(v)) = \mu \, \mathtt v_{1,0}
\end{equation}
for some $\mu \in \R$. 
Recalling the definition \eqref{def.ttv} of $\mathtt v_{1,0}$ 
and the fact that $\ph_{1,0}(x) = c x_3$ for some nonzero normalizing coefficient $c$, 
the second orthogonality relation in \eqref{eq:orthogonal.vectors} can be written as
\begin{equation}  \label{orthog.ttv.1.0}
\la \mF(\om(v), u(v)) , \mathtt v_{1,0} - p(v) \ra_{L^2(\S^2)} = 0,
\end{equation}
where $\| p(v) \|_{L^2(\S^2)} = O(|v|)$. 
Plugging \eqref{mF.mu.ttv.1.0} into \eqref{orthog.ttv.1.0} gives
\[
0 = \mu \la \mathtt v_{1,0} , \, \mathtt v_{1,0} - p(v) \ra_{L^2(\S^2)}
= \mu \big( |\mathtt v_{1,0}|^2 - \la \mathtt v_{1,0} , \, p(v) \ra_{L^2(\S^2)} \big)
= \mu ( 1 + O(|v|) ).
\]
The coefficient of $\mu$ is nonzero, and therefore $\mu = 0$. 
Then \eqref{mF.mu.ttv.1.0} gives \eqref{bif.eq.ZN.bis}.
\end{proof}

Applying Lemmas \ref{lemma:n.critical.orbits}, \ref{lemma:n=1}, \ref{lemma:sol.bif},
and defining $(\eta_a^{(i)}, \beta_a^{(i)}) := u^{(i)} := u( v^{(i)} )$ and 
$\om_a^{(i)} := \om (v^{(i)})$, $i = 1, \ldots, n$, 
the proof of Theorem \ref{thm:main} is complete.

\section{The symmetric case}

Given $u = (\eta, \beta) \in L^2(\S^2, \R) \times L^2(\S^2, \R)$, 
we define 
\begin{equation} \label{def.mY.2.3.23}
(\mY_2 u)(x) := \begin{pmatrix} 
\eta(x_1, -x_2, x_3) \\ 
- \beta(x_1, -x_2, x_3)
\end{pmatrix}, 
\quad 
(\mY_3 u)(x) := \begin{pmatrix} 
\eta(x_1, x_2, - x_3) \\ 
\beta(x_1, x_2, -x_3)
\end{pmatrix}, 
\quad 
\mY_{23} := \mY_2 \circ \mY_3
\end{equation}
for all $x \in \S^2$, 
and we note that $\mY_2 \circ \mY_3 = \mY_3 \circ \mY_2$ 
and $\mY_i^{-1} = \mY_i = \mY_i^T$, $i=2,3$, 
where $\mY_i^{-1}$ is the inverse operator, 
and $\mY_i^T$ is the transpose operator with respect to the 
$L^2(\S^2, \R) \times L^2(\S^2, \R)$ scalar product.
We define 
\begin{align} 
Y_2 & := \{ u \in L^2(\S^2, \R) \times L^2(\S^2, \R) : \mY_2 u = u \}, 
\notag \\
Y_3 & := \{ u \in L^2(\S^2, \R) \times L^2(\S^2, \R) : \mY_3 u = u \}, 
\quad \ 
Y_{23} := Y_2 \cap Y_3, 
\label{def.Y2.ecc}
\end{align}
so that $Y_3$ is the subspace of the functions $u$ that are even in $x_3$, 
and $Y_2$ is the subspace of the pairs $u = (\eta, \beta)$ where $\eta$ is even in $x_2$ 
and $\beta$ is odd in $x_2$. 
We denote  
\begin{equation}
V_2 := V \cap Y_2, \quad \ 
V_3 := V \cap Y_3, \quad \ 
V_{23} := V \cap Y_{23} = V_2 \cap V_3,
\end{equation}
and we use analogous notation for the spaces $V_D, V_N, W, R, Z$. 

\begin{lemma} 
One has 
\begin{equation}  \label{funct.mY}
\mH \circ \mY_i = \mH, \quad \ 
\mV \circ \mY_i = \mV, \quad \ 
\mI \circ \mY_i = \mI, \quad \ 
\mB_3 \circ \mY_i = - \mB_3, \quad \ 
i = 2, 3
\end{equation}
and 
\begin{alignat}{2}  
(\grad \mH) \circ \mY_i & = \mY_i \circ (\grad \mH), \quad \ &
(\grad \mV) \circ \mY_i & = \mY_i \circ (\grad \mV), 
\label{grad.mY.prima} \\
(\grad \mI) \circ \mY_i & = \mY_i \circ (\grad \mI), \quad \ & 
(\grad \mB_3) \circ \mY_i & = - \mY_i \circ (\grad \mB_3), \quad \ 
i = 2, 3.
\label{grad.mY.seconda}
\end{alignat}
As a consequence, 
\begin{equation}  \label{mF.mY}
\mF(\om, \mY_i u) = \mY_i \mF(\om, u).
\end{equation}
\end{lemma}

\begin{proof}
As is observed in Lemma 6.2 of \cite{B.J.LM}, 
given any $3 \! \times \! 3$ matrix $M$ such that $M M^T = I$, 
one has 
\[
G(\tilde \eta)\tilde \beta (x) = ( G(\eta)\beta )(Mx), 
\quad \ 
\grad_{\S^2} \tilde \eta (x) = M^T (\grad_{\S^2} \eta) (Mx), 
\quad \ 
| \grad_{\S^2} \tilde \eta (x)| = | (\grad_{\S^2} \eta) (Mx) | 
\]
for all $x \in \S^2$, all functions $\eta, \beta$, 
where $\tilde \eta(x) := \eta(Mx)$ and $\tilde \beta(x) := \beta(Mx)$. 
Hence the first two identities in \eqref{funct.mY} follows from the change of integration variable 
$Mx = y$, where $M = \operatorname{diag}(1, -1, 1)$ for $i=2$
and $M = \operatorname{diag}(1, 1, -1)$ for $i=3$.
To prove the third identity in \eqref{funct.mY}, we note that the operator $\mM$ in \eqref{def.mM} 
satisfies 
\[
\mM \circ \mY_2 = - \mY_2 \circ \mM, \quad \ 
\mM \circ \mY_3 = \mY_3 \circ \mM,
\]
and that $\mI(\eta, \beta)$ is linear in $\beta$. 
Finally, one has 
$\mB(\mY_2 u) = - M \mB(u)$ where $M = \operatorname{diag}(1, -1, 1)$, 
and $\mB (\mY_3 u) = M \mB(u)$ where $M = \operatorname{diag}(1, 1, -1)$.
Taking the third component of these vector identities, 
we obtain the fourth identity in \eqref{funct.mY}.
The identities in \eqref{grad.mY.prima}, \eqref{grad.mY.seconda} 
follow from \eqref{funct.mY} by taking the gradient and using the fact that 
the transpose operator $\mY_i^T$ is the inverse of $\mY_i$.
Identity \eqref{mF.mY} holds because $\mF(\om,u) = \grad \mH(u) - 2 \sigma_0 \grad \mV(u) 
- \om \grad \mI(u)$. 
\end{proof}

From formula \eqref{mF.mY} it follows that 
if $u \in Y_i$, then $\mF(\om,u) \in Y_i$, $i=2,3$, 
namely $Y_i$ is an invariant subspace for the nonlinear operator $\mF(\om, \cdot)$.

\begin{lemma}
\label{lemma:parity.x3.ph.ell.m}
The spherical harmonics $\ph_{\ell,m}$ in \eqref{def.ph.ell.m.Re.Im} satisfy
\[
\ph_{\ell,m} (x_1, x_2, - x_3) = (-1)^{\ell-m} \ph_{\ell,m}(x) 
\quad \ \forall (\ell, m) \in T
\]
and 
\begin{alignat*}{2}
\ph_{\ell,m} (x_1, - x_2, x_3) 
& = \ph_{\ell,m}(x),  
\quad & m & = 0, \ldots, \ell,
\\
\ph_{\ell,m} (x_1, - x_2, x_3) 
& = - \ph_{\ell,m}(x), 
\quad & m & = -\ell, \ldots, -1,
\end{alignat*}
for all $x = (x_1, x_2, x_3) \in \S^2$, 
all $\ell \in \N_0$.
As a consequence, the vectors $\mathtt{v}_{\ell,m}$ defined in \eqref{def.ttv} satisfy
\begin{equation} \label{ttv.mY}
\mY_2 \mathtt{v}_{\ell,m} = \sigma(m) \mathtt v_{\ell,m}, \quad \ 
\mY_3 \mathtt{v}_{\ell,m} = (-1)^{\ell-m} \mathtt v_{\ell,m} 
\quad \ \forall (\ell, m) \in S,
\end{equation}
where $\sigma(m) := 1$ if $m \geq 0$ 
and $\sigma(m) := -1$ if $m < 0$. 
\end{lemma}

From \eqref{ttv.mY} it follows that $\mY_i(V_i) = V_i$, $\mY_i(W_i) = W_i$, $i = 2,3$. 
We observe that, 
if a function $\eta$ is even in $x_3$, 
then only the spherical harmonics $\ph_{\ell,m}$ that are even in $x_3$ 
appear in its expansion, i.e., 
$\hat \eta_{\ell,m} = 0$ for all $(\ell,m) \in T$ such that $\ell - m$ is odd, 
and only coefficients $\hat \eta_{\ell,m}$ with $\ell - m$ even  
can be nonzero. 
Also, if $\eta$ is even in $x_2$ and $\beta$ is odd in $x_2$,  
then their coefficients satisfy
$\hat \eta_{\ell,m} = 0$ for $m < 0$ 
and $\hat \beta_{\ell,m} = 0$ for $m \geq 0$, 
and only coefficients $\hat \eta_{\ell,m}$ with $m \geq 0$ 
and $\hat \beta_{\ell,m}$ with $m < 0$ 
can be nonzero. 

On the subspace $Y_3$ one has the torus action $\mT_\th$, $\th \in \T$, 
because the matrices $R(\th)$ and $\mathrm{diag}(1, 1, -1)$ commute, and therefore 
\[
\mY_3 \circ \mT_\th = \mT_\th \circ \mY_3.
\]

\begin{proof}[Proof of Theorem \ref{thm:symmetric}]
$(i)$ The arguments of the proof of Theorem \ref{thm:main} can be repeated in the subspace $Y_3$.
$(ii)$ The arguments of the proof of Theorem \ref{thm:main} can be repeated in the subspace $Y_2$, 
except for the fact that we do not have the torus action $\mT_\th$, 
and therefore we do not obtain $n$ orbits of critical points of the restricted functional, 
but only two critical points, namely one minimum point and one maximum point for the functional. 
$(iii)$ is similar to $(ii)$. 
\end{proof}

\bigskip

\begin{flushright}

\textbf{Pietro Baldi}

Dipartimento di Matematica e Applicazioni ``R. Caccioppoli''

University of Naples Federico II

Via Cintia, Monte Sant'Angelo, 80126 Naples, Italy

pietro.baldi@unina.it

\medskip

\textbf{Domenico Angelo La Manna}

Dipartimento di Matematica e Applicazioni ``R. Caccioppoli''

University of Naples Federico II

Via Cintia, Monte Sant'Angelo, 80126 Naples, Italy

domenicoangelo.lamanna@unina.it

\medskip

\textbf{Giuseppe La Scala}

Mathematical and Physical Sciences for Advanced Materials and Technologies

Scuola Superiore Meridionale

Via Mezzocannone, 4, 80138 Naples, Italy

giuseppe.lascala-ssm@unina.it

\end{flushright}
\end{document}